\documentclass[12pt]{amsart}
\usepackage{amscd,amsmath,amsthm,amssymb,graphics}
\usepackage{pstcol,pst-plot,pst-3d}%\usepackage[T1]{fontenc}
\usepackage{lmodern,pst-node}%\usepackage{geometric}
\usepackage[dvips]{graphicx}
\usepackage{multicol}
\usepackage{epic,eepic}
\usepackage{amsfonts,amssymb,amscd,amsmath,enumitem,verbatim}
\psset{unit=0.7cm,linewidth=0.8pt,arrowsize=2.5pt 4}
\newpsstyle{fatline}{linewidth=1.5pt}
\newpsstyle{fyp}{fillstyle=solid,fillcolor=verylight}
\definecolor{verylight}{gray}{0.97}
\definecolor{light}{gray}{0.9}
\definecolor{medium}{gray}{0.85}
\definecolor{dark}{gray}{0.6}

\usepackage{lineno}

\def\NZQ{\Bbb}               % the font for N,Z,Q,R,C
\def\NN{{\NZQ N}}

\def\frk{\frak}               % font for "Fraktur"

\def\Phi{{\frk n}}
\def\Phi{{\frk N}}
\def\MP{{\mathcal P}}

\def\MT{{\mathcal T}}

\def\Cc{{\mathcal C}}

\def\Pc{{\mathcal P}}

\def\opn#1#2{\def#1{\operatorname{#2}}} % to make operators
\opn\chara{char} \opn\length{\ell} \opn\pd{pd} \opn\rk{rk}
\opn\projdim{proj\,dim} \opn\injdim{inj\,dim} \opn\rank{rank}
\opn\depth{depth} \opn\grade{grade} \opn\height{height}
\opn\embdim{emb\,dim} \opn\codim{codim}

\opn\Tr{Tr} \opn\bigrank{big\,rank}
\opn\superheight{superheight}\opn\lcm{lcm}
\opn\trdeg{tr\,deg}%\emph{
\opn\reg{reg} \opn\lreg{lreg} \opn\ini{in} \opn\lpd{lpd}
\opn\size{size}\opn\bigsize{bigsize}
\opn\cosize{cosize}\opn\bigcosize{bigcosize}
\opn\sdepth{sdepth}\opn\sreg{sreg}
\opn\link{link}\opn\fdepth{fdepth}
\opn\div{div} \opn\Div{Div} \opn\cl{cl} \opn\Cl{Cl}
\let\epsilon\varepsilon
\let\phi=\varphi
\let\kappa=\varkappa
\opn\Spec{Spec} \opn\Supp{Supp} \opn\supp{supp} \opn\Sing{Sing}
\opn\Ass{Ass} \opn\Min{Min}\opn\Mon{Mon} \opn\dstab{dstab} \opn\astab{astab}
\opn\Syz{Syz}
\opn\Ann{Ann} \opn\Rad{Rad} \opn\Soc{Soc}
\opn\Im{Im} \opn\Ker{Ker} \opn\Coker{Coker} \opn\Am{Am}
\opn\Hom{Hom} \opn\Tor{Tor} \opn\Ext{Ext} \opn\End{End}
\opn\Aut{Aut} \opn\id{id}

\opn\nat{nat}
\opn\pff{pf}%   \pf exists already
\opn\Pf{Pf} \opn\GL{GL} \opn\SL{SL} \opn\mod{mod} \opn\ord{ord}
\opn\Gin{Gin} \opn\Hilb{Hilb}\opn\sort{sort}
\opn\initial{init}
\opn\ende{end}
\opn\height{height}
\opn\type{type}
\opn\set{set}
\opn\Rook{Rook}
\opn\aff{aff} \opn\con{conv} \opn\relint{relint} \opn\st{st}
\opn\lk{lk} \opn\cn{cn} \opn\core{core} \opn\vol{vol}
\opn\link{link} \opn\star{star}\opn\lex{lex}
\opn\gr{gr}

\def\pot#1#2{#1[\kern-0.28ex[#2]\kern-0.28ex]}

\opn\dirlim{\underrightarrow{\lim}}
\opn\inivlim{\underleftarrow{\lim}}
\let\iso=\cong

\let\to=\rightarrow

\def\Implies{\ifmmode\Longrightarrow \else
        \unskip${}\Longrightarrow{}$\ignorespaces\fi}
\def\implies{\ifmmode\Rightarrow \else
        \unskip${}\Rightarrow{}$\ignorespaces\fi}
\def\iff{\ifmmode\Longleftrightarrow \else
        \unskip${}\Longleftrightarrow{}$\ignorespaces\fi}

\let\:=\colon
 \theoremstyle{plain}
\newtheorem{Theorem}{Theorem}[section]
 \newtheorem{Lemma}[Theorem]{Lemma}
 \newtheorem{Corollary}[Theorem]{Corollary}
 \newtheorem{Proposition}[Theorem]{Proposition}

 \theoremstyle{definition}

 \newtheorem{Observation}[Theorem]{Observation}
 
 \newtheorem{Example}[Theorem]{Example}

\let\epsilon\varepsilon
\let\kappa=\varkappa
\opn\dis{dis}
\def\pnt{{\raise0.5mm\hbox{\large\bf.}}}

\opn\Lex{Lex}
\newcommand{\rad}{1.5 pt}
\usepackage{subcaption,tikz}

\usepackage{float}
\begin{document}
\title{Regularity and Gorenstein property of the $L$-convex Polyominoes}
\author {Viviana Ene, J\"urgen Herzog,  Ayesha Asloob Qureshi, Francesco Romeo}
\address{Viviana Ene, Faculty of Mathematics and Computer Science, Ovidius University, Bd.\ Mamaia 124,  900527 Constanta, Romania}  \email{vivian@univ-ovidius.ro}
\address{J\"urgen Herzog, Fachbereich Mathematik, Universit\"at Duisburg-Essen, Fakult\"at f\"ur Mathematik, 45117 Essen, Germany} \email{juergen.herzog@uni-essen.de}
\address{Ayesha Asloob Qureshi, Sabanc\i \; University, Faculty of Engineering and Natural Sciences, Orta Mahalle, Tuzla 34956, Istanbul, Turkey}\email{aqureshi@sabanciuniv.edu}
\address{Francesco Romeo, University of Trento, Department of Mathematics, via Sommarive, 14, 38123 Povo (Trento), Italy}\email{francesco.romeo-3@unitn.it}

%\thanks{Part of this paper was written while the second author visited Faculty of Engineering and Natural Sciences at Sabanci University.}

\begin{abstract}
We study the coordinate ring of an $L$-convex polyomino, determine its regularity  in terms of the maximal number of rooks that can be placed in the polyomino. We also characterize the Gorenstein $L$-convex polyominoes and those which are Gorenstein on the punctured spectrum, and compute the Cohen--Macaulay type of any $L$-convex polyomino in terms of the maximal rectangles covering it. 
\end{abstract}

\thanks{This work was supported by The Scientific and Technological Research Council of Turkey - TUBITAK (Grant No: 118F169), and by the Doctoral School in Mathematics of the University of Trento.}

\maketitle
\section*{Introduction}

A classical subject of commutative algebra and algebraic geometry 
is the study of the ideal of $t$-minors and related ideals of an $m\times n$-matrix $X=(x_{ij})$ of indeterminates, see for example the lecture notes  \cite{BV} and its references to original articles. Several years after the appearance of the these lecture notes, a new aspect of the theory was introduced by considering Gr\"obner bases of determinantal ideals. These studies were initiated by the articles \cite{Na}, \cite{CGG} and \cite{Stu}. 

More generally and motivated by geometric applications, ideals  of $t$-minors of $2$-sided ladders have been studied, see \cite{CH},\cite{Co1},\cite{Co2} and \cite{Gor}.  For the case of $2$-minors,  these classes of ideals may be considered as special cases of the ideal $I_{\Pc}$ of inner $2$-minors of a polyominoe $\Pc$.  Ideals of this kind where first introduced and studied by Qureshi~\cite{Q}. For a field $K$, the $K$-algebra $K[\Pc]$ whose relations are given by $I_\Pc$ is called the coordinate ring of $\Pc$.  

Roughly speaking, a polyomino  is a plane figure  obtained by joining squares of equal sizes along their edges. For the precise definition we refer to the first section of this paper. A very nice introduction to the combinatorics of polyominoes and tilings is given in the monograph \cite{Tu}.

One of the most challenging problems in the algebraic theory of polyminoes is the classification  of  the  polyminoes $\Pc$ whose coordinate ring $K[\Pc]$ is a domain. It has been shown in \cite{HM} and \cite{QShShi} that this is the case if the polyomino is simply connected. In a more recent paper by Mascia, Rinaldo and  Romeo \cite{MRR}, it is shown that if  $K[\Pc]$ is a domain then it should have   no zig-zag walks, and they conjecture that this is also a sufficient condition for $K[\Pc]$ to be a domain. They verified  this conjecture  computationally for polyominoes of rank $\leq 14$. It is
expected that $K[\Pc]$ is always reduced.

Additional  structural results on $K[\Pc]$ for special classes of polyominoes were already shown in  Qureshi's article \cite{Q}. There she proved that $K[\Pc]$ is a Cohen--Macaulay normal domain when $\Pc$ is a convex polyomino, and characterized all stack polyominoes for which $K[\Pc]$ is Gorenstein by computing the  class group of this algebra. 

In the present paper we focus on so-called $L$-polyominoes. This is a particularly nice class of convex polyominoes which is distingushed  by the property that any two cells  of the polyomino can be connected by a path of cells with at most one change of directions. The combinatorics of this class of polyominoes is described the  paper \cite{CR}
and \cite{CFRR} by Castiglione et al. In Section 1 we recall some of the remarkable properties of $L$-polyominoes referring to the above mentioned papers. In particular, if $\Pc$ is an $L$-convex polyomino,  there is a natural bipartite graph $F_\Pc$ whose edges correspond to the cells of $\Pc$. By using this correspondence, we show in Proposition~\ref{prop:Ferrer} that there exists a polyomino $\Pc^*$ which is a Ferrer diagram and such that the bipartite graphs $F_\Pc$ and $F_{\Pc^*}$ are isomorphic. We call $\Pc^*$ the Ferrer diagram projected by $\Pc$.  Similarly there exists a bipartite graph $G_\Pc$ whose edges correspond to the coordinates of the  vertices of $\Pc$.  By using the intimate relationship between $F_\Pc$  and $G_\Pc$ it can be shown that $G_\Pc$ and $G_\Pc^*$ are isomorphic as well, see Corollary~\ref{cor:ferrer}. The crucial observation  which then follows from these considerations is the result (Theorem~\ref{thm:K[P]}) that $K[\Pc]$ and  $K[\Pc^*]$ are isomorphic as standard graded $K$-algebras.  Therefore  all algebraic invariants and properties of $K[\Pc]$ are shared by $K[\Pc^*]$. This  allows us for many arguments in our proofs to assume that $\Pc$ itself is a Ferrer diagram. Since the coordinate ring of a Ferrer diagram can be identified with the edge ring of a Ferrer graph, results of Corso and Nagel  \cite{CN} can be used  to compute the Castelnuovo-Mumford  regularity of $K[\Pc]$, denoted by $\reg(K[\Pc])$. It turns out that $\reg(K[\Pc])$ has a very nice combinatorial interpretation. Namely, for an $L$-convex polyomino,  $\reg(K[\Pc])$ is equal to maximal number of non-attacking rooks that can be placed on $\Pc$, as shown in Theorem~\ref{thm:rookreg}. This is the main result of Section 2. 

In Section~3 we study the Gorenstein property of $L$-convex polyominoes. We first observe that if we remove the rectangle of maximal width from $\Pc$, then the result is again an $L$-convex polyomino. Repeating this process we obtain a finite sequence of $L$-convex polyominoes, which we call the derived sequence of $\Pc$.  In Theorem~\ref{gorenstein} we then  shown that $K[\Pc]$ is Gorenstein if and only if the bounding boxes of the derived sequence of $L$-convex polyominoes of $\Pc$ are all squares. For the proof we use again that $K[\Pc] \iso K[\Pc^*]$, and the characterization of Gorenstein stack polyominoes given by Qureshi in \cite{Q}. In addition, under the assumption  $K[\Pc]$ is  not Gorenstein,  we show in Theorem~\ref{gorenstein} that $K[\Pc]$ is Gorenstein on the punctured spectrum if and only if  $\Pc$ is a rectangle, but not a square. Here we use that the coordinate ring of a a Ferrer diagram may be viewed as a Hibi ring. Then we may apply a recent result of Herzog et al \cite{HMP} which characterizes the Hibi rings which are Gorenstein on the punctured spectrum. 

Finally,  in Section~4 we compute the Cohen--Macaulay type of $K[\Pc]$ for  an $L$-convex polyomino $\Pc$. Again we use the fact that $K[\Pc^*]$ may be viewed as a Hibi ring (of a suitable poset $Q$). The number of generators of  the canonical module of $K[\Pc^*]$, which by definiton is the Cohen--Macaulay type, is described by Miyazaki \cite{Mi} (based on  results of Stanley \cite{St} and Hibi \cite{Hi}).  It is  the number of minimal  strictly order reversing maps on $Q$. Somewhat tedious counting arguments then  provide us in Theorem~\ref{thm:type} with the desired formula.

\section{Some combinatorics of $L$-convex polyominoes}
\subsection{Polyominoes}

In this subsection we recall definitions and notation about polyominoes. If $a = (i, j), b = (k, \ell) \in \NN^2$, with $i	\leq k$ and $j\leq\ell$, the set $[a, b]=\{(r,s) \in \NN^2 : i\leq r \leq k \text{ and } j \leq s \leq \ell\}$ is called an \textit{interval} of $\NN^2$. If $i<k$ and $j < \ell$, $[a,b]$ is called a \textit{proper interval}, and the elements $a,b,c,d$ are called corners of $[a,b]$, where $c=(i,\ell)$ and $d=(k,j)$. In particular, $a$ and $b$ are called the \textit{diagonal corners} whereas  $c$ and $d$ are called the \textit{anti-diagonal corners} of $[a,b]$. The corner $a$ (resp. $c$) is also called the lower left (resp. upper) corner of $[a,b]$, and $d$ (resp. $b$) is the right lower (resp. upper) corner of $[a,b]$.
A proper interval of the form $C = [a, a + (1, 1)]$ is called a \textit{cell}. Its vertices $V(C)$ are $a, a+(1,0), a+(0,1), a+(1,1)$ and its edges $E(C)$ are
\[
 \{a,a+(1,0)\}, \{a,a+(0,1)\},\{a+(1,0),a+(1,1)\},\{a+(0,1),a+(1,1)\}.
\]
Let $\MP$ be a finite collection of cells of $\NN^2$, and let $C$ and $D$ be two cells of $\MP$. Then $C$ and $D$ are said to be \textit{connected}, if there is a sequence of cells $C = C_1,\ldots, C_m = D$ of $\MP$ such that $C_i$ and $C_{i+1}$ have a common edge for all $i = 1,\ldots, m - 1$. In addition, if $C_i \neq C_j$ for all $i \neq j$, then $C_1,\dots, C_m$ is called a \textit{path} (connecting $C$ and $D$). A collection of cells $\MP$ is called a \textit{polyomino} if any two cells of $\MP$ are connected. We denote by $V(\MP)=\cup _{C\in \MP} V(C)$ the vertex set of $\MP$. A polyomino $\Pc'$ whose cells belong to $ \Pc$ is called a \emph{subpolyomino} of $\Pc$.

A polyomino $\MP$ is called \emph{row convex} if for any two cells with lower left corners $a=(i,j)$ and $b=(k,j)$, with $k>i$, it follows that all cells with lower left corners $(l,j)$ with $i\leq l \leq k$ are cells of $\MP$. Similarly,  $\MP$ is called \emph{column convex} if for any two cells with lower left corners $a=(i,j)$ and $b=(i,k)$, with $k>j$, it follows that all cells with lower left corners $(i,l)$ with $j\leq l \leq k$ are cells of $\MP$. If a polyomino $\Pc$ is simultaneously row and column convex, it is called just \emph{convex}.

Each proper interval  $[(i,j),(k,l)] $ in $\NN^2$ can be identified as a polyomino and it is referred to as {\em rectangular} polyomino. %The size of a rectangular polyomino $\Pc$ with $V(\Pc)=[(i,j),(k,l)] $ is defined to be $(l-j)+(k-i)+1$.
A rectangular subpolyomino $\Pc'$ of $\Pc$ is called \emph{maximal} if there is no rectangular subpolyomino $\Pc''$ of $\Pc$
that properly contains $\Pc'$.
Given a polyomino $\Pc$, the rectangle that contains $\Pc$ and has the smallest size with this property is called \emph{bounding box} of $\Pc$. After a shift of coordinates, we may assume that the bounding box is $[(0,0),(m,n)]$ for some $m,n \in \mathbb{N}$. In this case, the width of $\Pc$, denoted by $w(\Pc)$ is $m$. Similarly, the height of $\Pc$, denoted by $h(\Pc)$ is $n$.

Moreover, an interval $[a,b]$ with $a = (i,j)$ and $b = (k, \ell)$ is called a \textit{horizontal edge interval} of $\MP$ if $j =\ell$ and the sets $\{(r, j), (r+1, j)\}$ for $r = i, \dots, k-1$ are edges of cells of $\Pc$. If a horizontal edge interval of $\Pc$ is not
strictly contained in any other horizontal edge interval of $\Pc$, then we call it a \textit{maximal horizontal edge interval}. Similarly one defines vertical edge intervals and maximal vertical edge intervals of $\MP$.

\subsection{$L$-convex polyominoes}\label{subs:Lconv}
% A polyomino is called $l$-convex if there exists a path ...

%Consider the natural partial order on $\NN^2$ defined as follows: let $(a,b),(c,d) \in \NN^2$. Then $(a,b)\leq(c,d)$ if and only if $a\leq b$ and $c\leq d$.

Let $\mathcal{C}:C_1, C_2, \ldots, C_m$ be a path of cells and $(i_k, j_k)$ be the lower left corner of $C_k$ for $1 \leq k \leq m$. Then $\Cc$ has a change of direction at $C_k$ for some $2 \leq k \leq m-1$ if $i_{k-1}\neq i_{k+1}$ and $j_{k-1} \neq j_{k+1}$.

A convex polyomino $\Pc$ is called $k$-convex if any two cells in $\Pc$ can be connected by a path of cells in $\Pc$ with at most $k$ change of directions. The $1$-convex polyominoes are simply called $L$-convex polyomino.  
The next lemma gives information about the maximal rectangles of an $L$-convex polyomino \cite{CR}. The maximal rectangles of the polyomino in Figure \ref{fig:maxrect} are of sizes $7 \times 2$, $4\times 5$, $3\times 6$, $2 \times 7$ and $1 \times 10$.
\begin{Lemma}\label{thm:maxrect2}
A maximal rectangle of an $L$-convex polyomino $\Pc$ has a unique occurrence in $\Pc$.
\end{Lemma}

\begin{figure}[H]
  \centering
  \begin{subfigure}{0.5\textwidth}
  \centering
    \resizebox{0.6\textwidth}{!}{
  \begin{tikzpicture}

\draw[thick] (0,4) --  (0,6);
\draw[thick] (1,2) --  (1,7);
\draw[thick] (2,1) --  (2,8);
\draw[thick] (3,0) --  (3,10);
\draw[thick] (4,0) --  (4,10);
\draw[thick] (5,1) --  (5,7);
\draw[thick] (6,4) --  (6,6);
\draw[thick] (7,4) --  (7,6);

\draw[thick] (3,0) --  (4,0);
\draw[thick] (2,1) --  (5,1);
\draw[thick] (1,2) --  (5,2);
\draw[thick] (1,3) --  (5,3);
\draw[thick] (0,4) --  (7,4);
\draw[thick] (0,5) --  (7,5);
\draw[thick] (0,6) --  (7,6);
\draw[thick] (1,7) --  (5,7);
\draw[thick] (2,8) --  (4,8);
\draw[thick] (3,9) --  (4,9);
\draw[thick] (3,10) --  (4,10);

%\fill[fill=red, fill opacity=0.5] (0,4) -- (7,4)-- (7,6) -- (0,6);
%\fill[fill=blue, fill opacity=0.5] (3,0) -- (4,0)-- (4,10) -- (3,10);

   \end{tikzpicture}}
   \caption{An $L$-convex polyomino $\Pc$.}\label{fig: lcon}
   \end{subfigure}%
\begin{subfigure}{0.5\textwidth}
  \centering
    \resizebox{0.6\textwidth}{!}{
  \begin{tikzpicture}

\draw[thick] (0,4) --  (0,6);
\draw[thick] (1,2) --  (1,4);\draw[loosely dotted] (1,4) --  (1,6);\draw[thick] (1,6) --  (1,7);
\draw[thick] (2,1) --  (2,2);\draw[loosely dotted] (2,2) --  (2,7);\draw[thick] (2,7) --  (2,8);
\draw[thick] (3,0) --  (3,1);\draw[loosely dotted] (3,1) --  (3,8);\draw[thick] (3,8) --  (3,10);
\draw[thick] (4,0) --  (4,1);\draw[loosely dotted] (4,1) --  (4,7);\draw[thick] (4,7) --  (4,10);
\draw[thick] (5,1) --  (5,4);\draw[loosely dotted] (5,4) --  (5,6);\draw[thick] (5,6) --  (5,7);
\draw[thick] (7,4) --  (7,6);

\draw[thick] (3,0) --  (4,0);
\draw[thick] (2,1) --  (5,1);
\draw[thick] (1,2) --  (5,2);

\draw[thick] (0,4) --  (7,4);

\draw[thick] (0,6) --  (7,6);
\draw[thick] (1,7) --  (5,7);
\draw[thick] (2,8) --  (4,8);

\draw[thick] (3,10) --  (4,10);

%\fill[fill=red, fill opacity=0.5] (0,4) -- (7,4)-- (7,6) -- (0,6);
%\fill[fill=blue, fill opacity=0.5] (3,0) -- (4,0)-- (4,10) -- (3,10);

   \end{tikzpicture}}
   \caption{The maximal rectangles of $\Pc$.}
   \end{subfigure}
   \caption{The maximal rectangles of $\Pc$ }\label{fig:maxrect}
\end{figure}
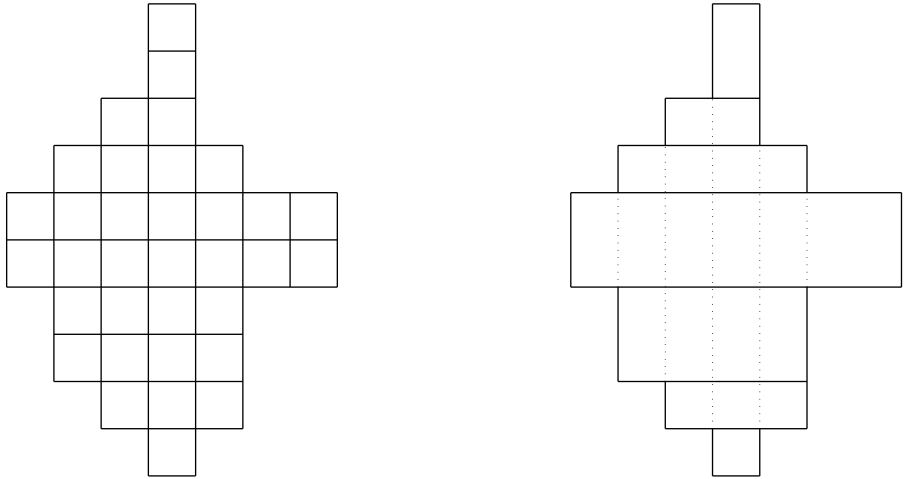
As a consequence of Lemma~\ref{thm:maxrect2} we have that, given an $L$-convex polyomino $\Pc$, there is a unique maximal rectangle $R_w$ such that $w(\Pc)=w(R_w)$ and a unique maximal interval $R_h$ such that $h(\Pc)=h(R_h)$.

\subsection{The bipartite graphs associated to polyominoes}
Let $\Pc$ be a convex polyomino with bounding box $[(0,0),(m,n)]$. In $\Pc$ there are $n$ rows of cells, numbered increasingly from the top to the bottom, and $m$ columns of cells, numbered increasingly from the left to the right. The cell in the intersection of $i$-th row and $j$-th column can be labelled as  $C_{ij}$. Therefore, we can attach a bipartite graph $F_{\Pc}$ to the polyomino $\Pc$ in the following way. Let $V(F_{\Pc})= \{X_{1},\ldots,X_{m}\} \sqcup \{Y_{1},\ldots,Y_{n}\}$ and
\[
E(F_{\Pc})=\{\{X_i,Y_j\} : C_{ij} \mbox{ is a cell of } \Pc\}.
\]
For all  $1\leq i\leq n$, we define the $i$-th horizontal projection of $\Pc$ as the number of cells in the $i$-th row, and denote it by $h_i$. Similarly, for all $1\leq j\leq m$, we define the $j$-th vertical projection of $\Pc$ as the number of cells in the $j$-th column and denote it by $v_j$. Note that $h_i=\deg Y_i$ and $v_j=\deg X_j$ in the graph $F_{\Pc}$. In the sequel, we will refer to the vector $H_{\Pc}=(h_1, h_2, \ldots, h_{n})$ as the horizontal projections of $\Pc$ and $V_{\Pc}=(v_1, v_2, \ldots, v_{m})$ as the vertical projection of $\Pc$. For an $L$-convex polyomino one has

\begin{Theorem}[{[1,Lemma 1,2,3 ]}] \label{thm:proj}
Let $\Pc$ be an $L$-convex polyomino, then:
\begin{enumerate}
\item[{\em (a)}]$\Pc$ is uniquely determined by $H_{\Pc}$ and $V_{\Pc}$;
\item[{\em (b)}] $H_{\Pc}$ and $V_{\Pc}$ are unimodal vectors;
\item[{\em (c)}] Let $j,j'$ be two different columns of cells of $\Pc$ such that $v_j \leq v_{j'}$. Then for each row $i$ of $\Pc$, we have $C_{ij'}\in \Pc$ if $C_{ij} \in \Pc$.
\item[{\em (d)}] Let $i,i'$ be two different rows of cells of $\Pc$ such that $h_i \leq h_{i'}$. Then for each column $j$ of $\Pc$, we have $C_{i'j}\in \Pc$ if $C_{ij} \in \Pc$.

\end{enumerate}
\end{Theorem}

\begin{figure}[H]
\centering
    \resizebox{0.3\textwidth}{!}{
  \begin{tikzpicture}

\draw[thick] (-1,1) --  (-1,2);
\draw[thick] (0,1) --  (0,3);
\draw[thick] (1,0) --  (1,5);
\draw[thick] (2,0) --  (2,5);
\draw[thick] (3,0) --  (3,5);
\draw[thick] (4,1) --  (4,2);

\draw[thick] (-1,1) --  (4,1);
\draw[thick] (-1,2) --  (4,2);
\draw[thick] (0,3) --  (3,3);
\draw[thick] (1,4) --  (3,4);
\draw[thick] (1,5) --  (3,5);
\draw[thick] (1,0) --  (3,0);

\fill[fill=gray, fill opacity=0.2] (1,2) -- (2,2)-- (2,3) -- (1,3);
\fill[fill=gray, fill opacity=0.2] (0,1) -- (0,3)-- (1,3) -- (1,1);
\fill[fill=gray, fill opacity=0.2] (2,0) -- (2,4)-- (3,4) -- (3,0);
\fill[fill=gray, fill opacity=0.2] (1,3) -- (1,5)-- (2,5) -- (2,3);
\fill[fill=gray, fill opacity=0.2] (3,1) -- (3,2)-- (4,2) -- (4,1);
\fill[fill=gray, fill opacity=0.2] (1,1) -- (1,2)-- (2,2) -- (2,1);
\fill[fill=gray, fill opacity=0.2] (1,0) -- (2,0)-- (2,1) -- (1,1);
\fill[fill=gray, fill opacity=0.2] (-1,1) -- (0,1)-- (0,2) -- (-1,2);
\fill[fill=gray, fill opacity=0.2] (2,4) -- (3,4)-- (3,5) -- (2,5);

   \end{tikzpicture}}
   \caption{An L-convex polyomino with $H_\Pc=(2,2,3,5,2)$ and \\ $V_\Pc=(1,2,5,5,1)$.}\label{fig:proj}
\end{figure}
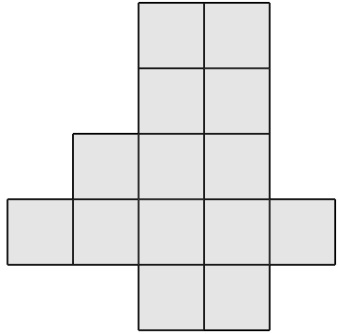
A {\em Ferrer graph} $G$ is a bipartite graph with $V(G)=\{u_1, \ldots, u_m\} \sqcup \{v_1, \ldots, v_n\}$ such that $\{u_1, v_n\}, \{u_m, v_1\} \in E(G)$ and if $\{u_i,v_j\} \in E(G)$ then $\{u_r,v_s\} \in E(G)$ for all $1 \leq r \leq i $ and for all $1 \leq s \leq j$.
Let $\Pc$ be a polyomino whose $F_{\Pc}$ is a Ferrer graph, then $\Pc$ is called a \emph{Ferrer diagram}. Note that a Ferrer diagram is a special type of stack polyomino (after a counterclockwise rotation by $90$ degrees).

\begin{figure}[H]
  \resizebox{0.3\textwidth}{!}{
  \begin{tikzpicture}

\draw (0,7)--(7,7);
\draw (0,6)--(7,6);
\draw (0,5)--(5,5);
\draw(0,4)-- (4,4);
\draw(0,3)-- (3,3);
\draw(0,2)-- (3,2);
\draw(0,1)-- (2,1);
\draw(0,0)-- (1,0);

\draw (0,0)--(0,7);
\draw (1,0)--(1,7);
\draw (2,1)--(2,7);
\draw (3,2)--(3,7);
\draw (4,4)-- (4,7);
\draw (5,5)-- (5,7);
\draw (6,6)-- (6,7);
\draw (7,6)-- (7,7);
%\fill[fill=gray, fill opacity=0.2] (1,2) -- (2,2)-- (2,3) -- (1,3);

   \end{tikzpicture}}
   \caption{Ferrer diagram}
\end{figure}
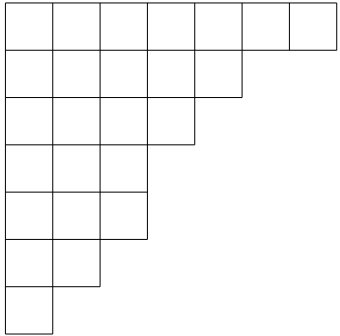
\begin{Proposition}\label{prop:Ferrer}
Let $\Pc$ be an $L$-convex polyomino. Then there exists a Ferrer diagram $\Pc^*$ such that $F_{\Pc} \iso F_{\Pc^*}$.
\end{Proposition}
\begin{proof}
Let $F_{\Pc}$ be the bipartite graph associated to $\Pc$, with vertex set $V(F_{\Pc})=\{X_{1},\ldots,X_{m}\}\sqcup \{Y_1,\ldots,Y_{n}\}$. We first prove that after a suitable relabelling of vertices of $F_\Pc$, it can be viewed as a Ferrer graph.
Let $T_{1}, T_{2}, \ldots, T_{m}$ and $U_{1}, U_{2}, \ldots, U_{n}$ be the relabelling of the vertices of $F_{\Pc}$ such that $\deg T_1\geq \deg T_2 \geq \cdots \geq \deg T_{m}$ and $\deg U_1\geq \deg U_2 \geq \cdots \geq \deg U_{n}$. We set  $v_{i}^*=\deg T_i$ for $1 \leq i \leq m$ and $h_j^*=\deg U_j$ for $1\leq j\leq n$.

%$v_{i}^*=\deg T_i$ for $1 \leq i \leq m-1$, $h_j^*=\deg U_j$ for $1\leq j\leq n-1$, and $v^*_{1} \geq v^*_{2} \geq \ldots \geq v^*_{m-1}$ and $h^*_{1} \geq  h^*_{2} \geq \ldots \geq  h^*_{n-1}$.
Then $v^*_{1}=n$ and  $h^*_{1}=m$ which implies that $\{T_{1},U_{n}\}, \{T_{m},U_{1}\}\in E(F_{\Pc})$.
Furthermore, let $\{T_{k},U_{l}\}\in E(F_{\Pc})$ for some $1 \leq k \leq m$ and $1 \leq l \leq n$. Then for all $1\leq r\leq k$ and $1\leq s \leq l$, we have $v^*_{k} \leq v^*_{r}$ and $h^*_{l} \leq h^*_{s}$. Therefore, by Theorem \ref{thm:proj}.(c), we see  that $\{T_{r},U_{s}\}\in E(F_{\Pc})$  for all $1\leq r\leq k$ and $1\leq s \leq l$.

Hence $F_{\Pc}$ is a Ferrer graph up to relabelling. Let $\Pc^*$ be the unique polyomino with horizontal and vertical projections $H_{\Pc^*}=(h^*_{1}, h^*_{2}, \ldots, h^*_{n})$ and $V_{\Pc^*}=(v^*_{1}, v^*_{2}, \ldots, v^*_{m})$, then $\Pc^*$ is a Ferrer diagram and $F_{\Pc} \iso F_{\Pc^*}$.
\end{proof}
From the proof of the above proposition, one sees that given an $L$-convex polyomino $\Pc$, the Ferrer diagram $\Pc^*$ such that $F_{\Pc} \iso F_{\Pc^*}$ is uniquely determined. We refer to $\Pc^*$ as the Ferrer diagram projected by $\Pc$.

  \begin{figure}[H]
  \centering
  \begin{subfigure}{0.455\textwidth}
   \centering
    \resizebox{0.5\textwidth}{!}{
  \begin{tikzpicture}
\draw[thick] (-1,1) --  (-1,2);
\draw[thick] (0,1) --  (0,3);
\draw[thick] (1,0) --  (1,5);
\draw[thick] (2,0) --  (2,5);
\draw[thick] (3,0) --  (3,4);
\draw[thick] (4,1) --  (4,2);

\draw[thick] (-1,1) --  (4,1);
\draw[thick] (-1,2) --  (4,2);
\draw[thick] (0,3) --  (3,3);
\draw[thick] (1,4) --  (3,4);
\draw[thick] (1,5) --  (2,5);
\draw[thick] (1,0) --  (3,0);

\fill[fill=gray, fill opacity=0.2] (1,2) -- (2,2)-- (2,3) -- (1,3);
\fill[fill=gray, fill opacity=0.2] (0,1) -- (0,3)-- (1,3) -- (1,1);
\fill[fill=gray, fill opacity=0.2] (2,0) -- (2,4)-- (3,4) -- (3,0);
\fill[fill=gray, fill opacity=0.2] (1,3) -- (1,5)-- (2,5) -- (2,3);
\fill[fill=gray, fill opacity=0.2] (3,1) -- (3,2)-- (4,2) -- (4,1);
\fill[fill=gray, fill opacity=0.2] (1,1) -- (1,2)-- (2,2) -- (2,1);
\fill[fill=gray, fill opacity=0.2] (1,0) -- (2,0)-- (2,1) -- (1,1);
\fill[fill=gray, fill opacity=0.2] (-1,1) -- (0,1)-- (0,2) -- (-1,2);
   \end{tikzpicture}}
   \caption{$L$-convex polyomino $\Pc$}
   \end{subfigure}%
  \begin{subfigure}{0.455\textwidth}
   \centering
    \resizebox{0.5\textwidth}{!}{
  \begin{tikzpicture}
  \draw[thick] (-1,0) --  (-1,5);
 \draw[thick] (0,0) --  (0,5);
\draw[thick] (1,1) --  (1,5);
\draw[thick] (2,3) --  (2,5);
\draw[thick] (3,4)--(3,5);
\draw[thick] (4,4)--(4,5);

\draw[thick] (-1,0)-- (0,0);
\draw[thick] (-1,1) --  (1,1);
\draw[thick] (-1,2) --  (1,2);
\draw[thick] (-1,3) --  (2,3);
\draw[thick] (-1,4) --  (4,4);
\draw[thick] (-1,5) --  (4,5);

\fill[fill=gray, fill opacity=0.2] (-1,4) -- (4,4)-- (4,5) -- (-1,5);
\fill[fill=gray, fill opacity=0.2] (-1,3) -- (2,3)-- (2,4) -- (-1,4);
\fill[fill=gray, fill opacity=0.2] (-1,1) -- (1,1)-- (1,3) -- (-1,3);
\fill[fill=gray, fill opacity=0.2] (-1,0) -- (0,0)-- (0,1) -- (-1,1);
   \end{tikzpicture}}
   \caption{The Ferrer diagram $\Pc^*$ projected by~$\Pc$}
   \end{subfigure}
   \end{figure}
Let $r(\Pc,k)$ be the number of ways of arranging $k$ non-attacking rooks in cells of $\Pc$. %The polynomial
%\[\varrho (\Pc)=\sum_{k}(-1)^k r(\Pc,k)x^{m+n-2k}\] is called the {\em rook polynomial} of polyomino $\Pc$.
Recall that, for a graph $G$ with $n$ vertices, a $k$-matching of $G$ is the  set of $k$ pairwise disjoint edges in $G$. Let $p(G,k)$ be the number of $k$ matchings of $G$.  It is a fact, for example see \cite[page 56]{GG}, that   $r(\Pc,k)= p(F_{\Pc},k)$.  Let $r(\Pc)$ denote the maximum number of rooks that can be arranged in $\Pc$ in non-attacking position, that is $r(\Pc)=\max_{k} r(\Pc,k)$. We have the following
\begin{Lemma}\label{lem:rooks}
Let $\Pc$ be an $L$-convex polyomino and $P^*$ be the Ferrer diagram projected by $\Pc$. Then $r(\Pc,k)=r(\Pc^*,k)$. In particular, $r(\Pc)=r(\Pc^*)$.
%Then The maximum number of non-attacking rooks for $\Pc$ and $\Pc^*$ are same.
\end{Lemma}
\begin{proof}
 From Proposition \ref{prop:Ferrer}, we have $F_\Pc \iso F_{\Pc^*}$ then $p(F_{\Pc},k)= p(F_{\Pc^*},k)$. Then by using the theorem on \cite[page 56]{GG}, we see that  $r(\Pc,k)=r(\Pc^*,k)$.
\end{proof}

 \begin{figure}[H]
  \centering
  \begin{subfigure}{0.5\textwidth}
   \centering
    \resizebox{0.6\textwidth}{!}{
  \begin{tikzpicture}

\draw[thick] (-1,1) --  (-1,2);
\draw[thick] (0,1) --  (0,3);
\draw[thick] (1,0) --  (1,5);
\draw[thick] (2,0) --  (2,5);
\draw[thick] (3,0) --  (3,4);
\draw[thick] (4,1) --  (4,2);

\draw[thick] (-1,1) --  (4,1);
\draw[thick] (-1,2) --  (4,2);
\draw[thick] (0,3) --  (3,3);
\draw[thick] (1,4) --  (3,4);
\draw[thick] (1,5) --  (2,5);
\draw[thick] (1,0) --  (3,0);

\fill[fill=gray, fill opacity=0.2] (1,2) -- (2,2)-- (2,3) -- (1,3);
\fill[fill=gray, fill opacity=0.2] (0,1) -- (0,3)-- (1,3) -- (1,1);
\fill[fill=gray, fill opacity=0.2] (2,0) -- (2,4)-- (3,4) -- (3,0);
\fill[fill=gray, fill opacity=0.2] (1,3) -- (1,5)-- (2,5) -- (2,3);
\fill[fill=gray, fill opacity=0.2] (3,1) -- (3,2)-- (4,2) -- (4,1);
\fill[fill=gray, fill opacity=0.2] (1,1) -- (1,2)-- (2,2) -- (2,1);
\fill[fill=gray, fill opacity=0.2] (1,0) -- (2,0)-- (2,1) -- (1,1);
\fill[fill=gray, fill opacity=0.2] (-1,1) -- (0,1)-- (0,2) -- (-1,2);

\node[inner sep=0pt] (rookl) at (-0.5,1.5)
    {\includegraphics[scale=0.6]{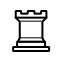}};
\node[inner sep=0pt] (rookl) at (0.5,2.5)
    {\includegraphics[scale=0.6]{rook.png}};
    \node[inner sep=0pt] (rookl) at (1.5,3.5)
    {\includegraphics[scale=0.6]{rook.png}};
    \node[inner sep=0pt] (rookl) at (2.5,0.5)
    {\includegraphics[scale=0.6]{rook.png}};
   \end{tikzpicture}}
   %\caption{maximum number of rooks in non-attacking position in $\Pc_1$ and $P_1^*$.}
   \end{subfigure}\hfill%
  \begin{subfigure}{0.5\textwidth}
  \centering
    \resizebox{0.6\textwidth}{!}{
  \begin{tikzpicture}
 
 \draw[thick] (-1,0) --  (-1,5);
 \draw[thick] (0,0) --  (0,5);
\draw[thick] (1,1) --  (1,5);
\draw[thick] (2,3) --  (2,5);
\draw[thick] (3,4)--(3,5);
\draw[thick] (4,4)--(4,5);

\draw[thick] (-1,0)-- (0,0);
\draw[thick] (-1,1) --  (1,1);
\draw[thick] (-1,2) --  (1,2);
\draw[thick] (-1,3) --  (2,3);
\draw[thick] (-1,4) --  (4,4);
\draw[thick] (-1,5) --  (4,5);

\fill[fill=gray, fill opacity=0.2] (-1,4) -- (4,4)-- (4,5) -- (-1,5);
\fill[fill=gray, fill opacity=0.2] (-1,3) -- (2,3)-- (2,4) -- (-1,4);
\fill[fill=gray, fill opacity=0.2] (-1,1) -- (1,1)-- (1,3) -- (-1,3);
\fill[fill=gray, fill opacity=0.2] (-1,0) -- (0,0)-- (0,1) -- (-1,1);

\node[inner sep=0pt] (rook) at (-0.5,1.5)
    {\includegraphics[scale=0.6]{rook.png}};
\node[inner sep=0pt] (rook) at (0.5,2.5)
    {\includegraphics[scale=0.6]{rook.png}};
    \node[inner sep=0pt] (rook) at (1.5,3.5)
    {\includegraphics[scale=0.6]{rook.png}};
    \node[inner sep=0pt] (rook) at (2.5,4.5)
    {\includegraphics[scale=0.6]{rook.png}};
   \end{tikzpicture}}
  
   \end{subfigure}
      \caption{Placement of rooks in non-attacking position in $\Pc$ and $\Pc^*$.}
   \end{figure}

As described in \cite[Section 2]{Q}, we can associate another bipartite graph $G_{\Pc}$ to $\Pc$ in the following way. Let $\mathcal{I}=[(0,0),(m,n)] $ be the bounding box of $\Pc$.
%{\red In $\MI$, there are $m+1$ coulmns of vertices (vertical intervals) and $n+1$ rows of vertices (horizontal intervals). We number the rows  in an increasing order from left to right and we number the rows  in an increasing order from top to bottom. Let $R_1, \ldots, R_n$ denote the rows and $C_1, \ldots, C_m$ denotes the columns of vertices of $\Pc$.}
Since $\Pc$ is convex, in $\Pc$  there are $m+1$ maximal vertical edge intervals and $n+1$ maximal horizontal edge intervals, namely there are $m+1$ columns and $n+1$ rows of vertices. We number the rows  in an increasing order from left to right and we number the rows  in an increasing order from top to bottom. Let $H_0, \ldots, H_n$ denote the rows and $V_0, \ldots, V_m$ denote the columns of vertices of $\Pc$.
Set $V(G_{\Pc})= \{x_0, \ldots, x_m\} \sqcup \{y_0, \ldots, y_n\}$. Then $\{x_i,y_j\} \in E(G_{\Pc})$ if and only if $V_i \cap H_j \neq \emptyset$. To distinguish between $G_{\Pc}$ and $F_{\Pc}$, we refer to them as follows:
\begin{itemize}
\item The graph $F_{\Pc}$ is the graph associated to the cells of $\Pc$.
\item The graph $G_{\Pc}$ is the graph associated to the vertices of $\Pc$.
\end{itemize}
\begin{figure}[H]
   \centering
    \resizebox{0.4\textwidth}{!}{
  \begin{tikzpicture}

\draw[thick] (-1,1) --  (-1,2);
\draw[thick] (0,1) --  (0,3);
\draw[thick] (1,0) --  (1,5);
\draw[thick] (2,0) --  (2,5);
\draw[thick] (3,0) --  (3,5);
\draw[thick] (4,1) --  (4,2);

\draw[thick] (-1,1) --  (4,1);
\draw[thick] (-1,2) --  (4,2);
\draw[thick] (0,3) --  (3,3);
\draw[thick] (1,4) --  (3,4);
\draw[thick] (1,5) --  (3,5);
\draw[thick] (1,0) --  (3,0);

\fill[fill=gray, fill opacity=0.2] (1,2) -- (2,2)-- (2,3) -- (1,3);
\fill[fill=gray, fill opacity=0.2] (0,1) -- (0,3)-- (1,3) -- (1,1);
\fill[fill=gray, fill opacity=0.2] (2,0) -- (2,4)-- (3,4) -- (3,0);
\fill[fill=gray, fill opacity=0.2] (1,3) -- (1,5)-- (2,5) -- (2,3);
\fill[fill=gray, fill opacity=0.2] (3,1) -- (3,2)-- (4,2) -- (4,1);
\fill[fill=gray, fill opacity=0.2] (1,1) -- (1,2)-- (2,2) -- (2,1);
\fill[fill=gray, fill opacity=0.2] (1,0) -- (2,0)-- (2,1) -- (1,1);
\fill[fill=gray, fill opacity=0.2] (-1,1) -- (0,1)-- (0,2) -- (-1,2);
\fill[fill=gray, fill opacity=0.2] (2,4) -- (3,4)-- (3,5) -- (2,5);

\node at (-1,-0.3){$x_0$};
\node at (0,-0.3){$x_1$};
\node at (1,-0.3){$x_2$};
\node at (2,-0.3){$x_3$};
\node at (3,-0.3){$x_4$};
\node at (4,-0.3){$x_5$};

\node at (-1.5,0){$y_5$};
\node at (-1.5,1){$y_4$};
\node at (-1.5,2){$y_3$};
\node at (-1.5,3){$y_2$};
\node at (-1.5,4){$y_1$};
\node at (-1.5,5){$y_0$};

\node at (-0.5,-0.3+6){$X_1$};
\node at (0.5,-0.3+6){$X_2$};
\node at (1.5,-0.3+6){$X_3$};
\node at (2.5,-0.3+6){$X_4$};
\node at (3.5,-0.3+6){$X_5$};
%\node at (4,-0.3){$x_6$};

\node at (-1.5+6.5,0.5){$Y_5$};
\node at (-1.5+6.5,1.5){$Y_4$};
\node at (-1.5+6.5,2.5){$Y_3$};
\node at (-1.5+6.5,3.5){$Y_2$};
\node at (-1.5+6.5,4.5){$Y_1$};
   \end{tikzpicture}}
   \caption{The two labellings on $\Pc$ of Figure \ref{fig:proj}}\label{fig : label}
\end{figure}
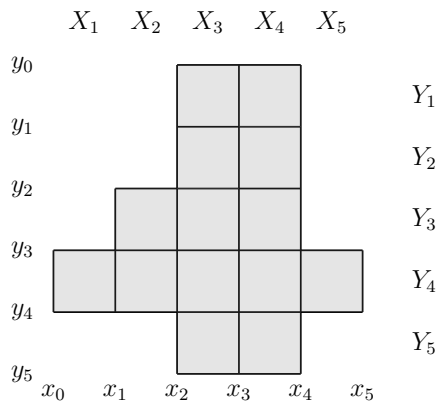
  \begin{figure}[H]
    \centering
    \begin{subfigure}{0.5\textwidth}
   \centering
    \resizebox{0.8\textwidth}{!}{
  \begin{tikzpicture}

\filldraw (-1,-0.3)  circle (\rad) node[anchor=north] {$x_0$};
\filldraw (0,-0.3) circle (\rad) node[anchor=north]{$x_1$};
\filldraw (1,-0.3)circle (\rad) node[anchor=north]{$x_2$};
\filldraw(2,-0.3)circle (\rad) node[anchor=north]{$x_3$};
\filldraw(3,-0.3)circle (\rad) node[anchor=north]{$x_4$};
\filldraw (4,-0.3)circle (\rad) node[anchor=north]{$x_5$};

\filldraw (-1,2.3)  circle (\rad) node[anchor=south] {$y_0$};
\filldraw (0,2.3) circle (\rad) node[anchor=south]{$y_1$};
\filldraw (1,2.3)circle (\rad) node[anchor=south]{$y_2$};
\filldraw(2,2.3)circle (\rad) node[anchor=south]{$y_3$};
\filldraw(3,2.3)circle (\rad) node[anchor=south]{$y_4$};
\filldraw (4,2.3)circle (\rad) node[anchor=south]{$y_5$};

\draw (-1,-0.3)--(2,2.3) ;
\draw (-1,-0.3)--(3,2.3) ;
\draw (0,-0.3)--(1,2.3) ;
\draw (0,-0.3)--(2,2.3) ;
\draw (0,-0.3)--(3,2.3) ;
\draw (1,-0.3)--(-1,2.3);
\draw (1,-0.3)--(0,2.3) ;
\draw (1,-0.3)--(1,2.3) ;
\draw (1,-0.3)--(2,2.3) ;
\draw (1,-0.3)--(3,2.3) ;
\draw (1,-0.3)--(4,2.3) ;
\draw (2,-0.3)--(-1,2.3);
\draw (2,-0.3)--(0,2.3) ;
\draw (2,-0.3)--(1,2.3) ;
\draw (2,-0.3)--(2,2.3) ;
\draw (2,-0.3)--(3,2.3) ;
\draw (2,-0.3)--(4,2.3) ;
\draw (3,-0.3)--(-1,2.3);
\draw (3,-0.3)--(0,2.3) ;
\draw (3,-0.3)--(1,2.3) ;
\draw (3,-0.3)--(2,2.3) ;
\draw (3,-0.3)--(3,2.3) ;
\draw (3,-0.3)--(4,2.3) ;
\draw (4,-0.3)--(2,2.3) ;
\draw (4,-0.3)--(3,2.3) ;
   \end{tikzpicture}}
   \caption{The bipartite graph $G_\Pc$ of the \\ polyomino $\Pc$ in Figure \ref{fig:proj}.}\label{fig : degc}
   \end{subfigure}%
      \begin{subfigure}{0.5\textwidth}

    \resizebox{0.65\textwidth}{!}{
  \begin{tikzpicture}

\filldraw (-1,-0.3)  circle (\rad) node[anchor=north] {$X_1$};
\filldraw (0,-0.3) circle (\rad) node[anchor=north]{$X_2$};
\filldraw (1,-0.3)circle (\rad) node[anchor=north]{$X_3$};
\filldraw(2,-0.3)circle (\rad) node[anchor=north]{$X_4$};
\filldraw(3,-0.3)circle (\rad) node[anchor=north]{$X_5$};

\filldraw (-1,2.3)  circle (\rad) node[anchor=south] {$Y_1$};
\filldraw (0,2.3) circle (\rad) node[anchor=south]{$Y_2$};
\filldraw (1,2.3)circle (\rad) node[anchor=south]{$Y_3$};
\filldraw(2,2.3)circle (\rad) node[anchor=south]{$Y_4$};
\filldraw(3,2.3)circle (\rad) node[anchor=south]{$Y_5$};

\draw (-1,-0.3)--(2,2.3) ;

\draw (0,-0.3)--(2,2.3) ;
\draw (0,-0.3)--(1,2.3) ;

\draw (1,-0.3)--(-1,2.3);
\draw (1,-0.3)--(0,2.3) ;
\draw (1,-0.3)--(1,2.3) ;
\draw (1,-0.3)--(2,2.3) ;
\draw (1,-0.3)--(3,2.3) ;

\draw (2,-0.3)--(-1,2.3);
\draw (2,-0.3)--(0,2.3) ;
\draw (2,-0.3)--(1,2.3) ;
\draw (2,-0.3)--(2,2.3) ;
\draw (2,-0.3)--(3,2.3) ;

\draw (3,-0.3)--(2,2.3) ;

   \end{tikzpicture}}
   \caption{The bipartite graph $F_{\Pc}$ of the \\ polyomino $\Pc$ in Figure \ref{fig:proj}.}\label{fig : degv}
   \end{subfigure}%
   \caption{}\label{fig:deg}
   \end{figure}

The relation between $F_\Pc$ and $G_\Pc$ is deducible from the following
\begin{Observation}\label{degreeprojection}
  Let $\Pc$ be an $L$-convex polyomino with bounding box $[(0,0),(m,n)]$, then the $H_{\Pc}$ has $n$ components and $V_{\Pc}$ has $m$ components. One can interpret both projections in terms of degrees of vertices of $G_{\Pc}$ in the following way:

(i) Let $V_{\Pc}=(v_1, v_2, \ldots, v_{m})$ and $H_{\Pc}=(h_1, h_2, \ldots, h_{n})$. From Theorem~\ref{thm:proj}, we know that $V_{\Pc}$ and $H_{\Pc}$ are unimodal. Let
\[
v_1\leq v_2\leq \cdots < v_i=n \geq v_{i+1} \geq \cdots \geq v_{m}
\]
for some $1 \leq i \leq m$. Then $v_j=\deg x_{j-1}-1$ for all $1 \leq j < i $ and $v_{j}=\deg x_{j}-1$ for $i \leq j \leq m$.
Similarly,  by using unimodality of $H_{\Pc}$, we get
\[
h_1\leq h_2\leq \cdots < h_i=m \geq h_{i+1} \geq \cdots \geq h_{n}
\]
for some $1 \leq i \leq m$. Then $h_j=\deg y_{j-1}-1$ for all $1 \leq j < i $ and $h_{j}=\deg y_{j}-1$ for $i \leq j \leq n$.

(ii) As a consequence of (i), if $\Pc$ is a Ferrer diagram with $V_{\Pc}=(v_1, v_2, \ldots, v_{m})$ and $H_{\Pc}=(h_1, h_2, \ldots, h_{n})$, then

\[
v_1\geq v_{2} \geq \cdots \geq v_{m},
\]
\[
h_1\geq h_2 \geq \cdots \geq h_{n}.
\]
Let $G_{\Pc}$ and $F_{\Pc}$ be the graphs associated to $\Pc$ as described above with $V(G_{\Pc})= \{x_0, \ldots, x_m\} \sqcup \{y_0, \ldots, y_n\}$ and $V(F_{\Pc})= \{X_1, \ldots, X_{m}\} \sqcup \{Y_1, \ldots, Y_{n}\}$.
Then $v_j=\deg X_j =\deg x_j-1$ for all $1 \leq j \leq m $, and $h_j=\deg Y_j= \deg y_{j}-1$ for all $1 \leq j \leq n$.
\end{Observation}

Now we obtain
\begin{Lemma}\label{degreeproperty}
Let $\Pc$ be an $L$-convex polyomino and $G_{\Pc}$ be its associated bipartite graph with $V(G_{\Pc})= \{x_0, \ldots, x_m\} \sqcup \{y_0, \ldots, y_n\}$. Then we have the following:
\begin{enumerate}
\item[\emph{(a)}] if $\deg x_i < \deg x_{i'}$, then $\{x_{i'}, y_j\} \in E(G_{\Pc})$ whenever $\{x_{i}, y_j\} \in E(G_{\Pc})$.
\item[\emph{(b)}] if $\deg y_j < \deg y_{j'}$, then $\{x_{i}, y_{j'}\} \in E(G_{\Pc})$ whenever $\{x_{i}, y_j\} \in E(G_{\Pc})$.
\end{enumerate}
\end{Lemma}

\begin{proof}
Let $H_{\Pc}=(h_1, h_2, \ldots, h_{n})$ and $V_{\Pc}=(v_1, v_2, \ldots, v_{m})$  be  the horizontal and vertical projection of $\Pc$.\\
(a): Let $p=\deg x_i < \deg x_{i'}=q$. Then following Observation~\ref{degreeprojection}, $h_s=p-1$ and $h_t=q-1$ for some $1 \leq s \neq t \leq n$.  Then $h_s< h_t$ and the conclusion follows from Theorem~\ref{thm:proj}(d).\\
(b): Let $p=\deg y_i < \deg y_{i'}=q$. Then following Observation~\ref{degreeprojection}, $v_s=p-1$ and $v_t=q-1$ for some $1 \leq s \neq t \leq m$.  Then $v_s< v_t$ and the the conclusion follows from Theorem~\ref{thm:proj}(c).
\end{proof}
A result similar to Proposition \ref{prop:Ferrer} holds also for the graph $G_\Pc$.

\begin{Corollary}\label{cor:ferrer}
Let $\Pc$ be an $L$-convex polyomino, let $\Pc^*$ be the Ferrer diagram projected by $\Pc$. Then $G_\Pc \iso G_{\Pc^*}$.
\end{Corollary}

\begin{proof}
First, we will show that $G_{\Pc}\iso H$ where $H$ is a Ferrer graph. Let $V(G_{\Pc})= \{x_0, \ldots, x_m\} \sqcup \{y_0, \ldots, y_n\}$. Similar to the proof of Proposition \ref{prop:Ferrer}, we relabel the vertices of $G_\Pc$ as $V(G_{\Pc})= \{t_0, \ldots, t_m\} \sqcup \{s_0, \ldots, s_n\}$ such that $\deg t_{0} \geq \deg t_{1} \geq \cdots \geq \deg t_{m}$ and $\deg s_{0} \geq \deg s_{1} \geq \cdots \geq \deg s_{n}$. Let $H$ be the new graph obtained by relabelling of the vertices of $G_{\Pc}$. Then by using Lemma \ref{degreeproperty}, we conclude that $H$ is a  Ferrer graph.

Now we will show that $H \iso G_{\Pc^*}$. This is an immediate consequence of Observation~\ref{degreeprojection}(ii). Indeed $V_{\Pc^*}=(\deg t_1-1, \ldots, \deg t_m-1)$ and \\$H_{\Pc^*}=(\deg s_1-1, \ldots, \deg s_n-1)$.
\end{proof}

%\begin{proof}
%Let $G_{\Pc}$ be the bipartite graph associated to $\Pc$, with vertex set $V(G_{\Pc})=\{x_{1},\ldots,x_m\}\sqcup \{y_1,\ldots,y_n\}$.
%Let $x_{i_1}, x_{i_2}, \ldots, x_{i_m}$ and $y_{j_1}, y_{j_2}, \ldots, y_{j_n}$ be the ordering of the vertices of $G_{\Pc}$ such that $\deg x_{i_1} \geq \deg x_{i_2} \geq \ldots \geq \deg x_{i_m}$ and $\deg y_{j_1} \geq \deg y_{j_2} \geq \ldots \geq \deg y_{j_n}$.
%Hence $\deg x_{i_1}=h(\Pc)+1=n$ and $\{x_{i_1},y_{j_n}\} \in E(G_{\Pc})$. Similarly, we have $\{x_{i_m},y_{j_1}\}\in E(G_{\Pc})$.
%Furthermore, by Lemma~\ref{degreeproperty}, we see that if $\{x_{i_k},y_{j_l}\}\in E(G_{\Pc})$ for some $1 \leq k \leq m$ and $1 \leq l \leq n$, then for all $1\leq r\leq k$ and $\leq s \leq l$, since $\deg x_{i_k} \leq \deg x_{i_r}$ and $\deg y_{j_l} \leq \deg y_{j_s}$, then $\{x_{i_r},y_{j_s}\}\in E(G_{\Pc})$.
%\end{proof}

%A result that will be useful in the sequel is the following \cite[Lemma 2]{CFRR}, according to our notation
%\begin{Lemma}
%Let $\Pc$ be an $L$-convex polyomino. Let $V_j,V_{j'} \in V(H_{\Pc})$ be such that $v_j\leq v_{j'}$. Then for any $H_{i} \in V(H_\Pc)$, if $\{H_i,V_j\} \in E(H_{\Pc})$, then $\{H_i,V_{j'}\} \in E(H_{\Pc})$.
%\end{Lemma}

 \section{regularity of $L$-convex polyominoes}
 
Let $K$ be a field. We denote by $S$ the polynomial over $K$ with variables $x_v$, where $v \in V (\Pc)$. The binomial $x_a x_b - x_c x_d\in S$ is called an \emph{inner 2-minor} of $\Pc$ if $[a,b]$ is a rectangular subpolyomino of $\Pc$. Here $c,d$ are the anti-diagonal corners of $[a,b]$. The ideal $I_\Pc \subset S$, generated by all of the inner 2-minors of $\Pc$, is called the \emph{polyomino ideal} of $\Pc$, and $K [\Pc] = S/I_\Pc$ is called the \emph{coordinate ring} of $\Pc$.

 \begin{Theorem}\label{thm:K[P]}
Let $\Pc$ be an $L$-convex polyomino and let $\Pc^*$ be  the Ferrer diagram projected by $\Pc$. Then $K[\Pc]$ and $K[\Pc^*]$ are isomorphic standard graded $K$-algebras.
\end{Theorem}
\begin{proof}
Since $\Pc$ is convex, it is known that $K[\Pc]$ is isomorphic to the edge ring $K[G_\Pc]$ of the bipartite graph $G_\Pc$  (see \cite[Section 2]{Q}).
By Corollary \ref{cor:ferrer}, $G_\Pc$ is isomorphic to $G_{\Pc^*}$. Hence the assertion follows.
\end{proof}
 \begin{Theorem}\label{regP}
 Let $\Pc$ be an $L$-convex polyomino and let $\Pc^*$ be the Ferrer diagram projected by $\Pc$. Moreover, let $H_{\Pc^*}=(h_1,\ldots, h_{n})$ . Then
 \[
 \reg (K[\Pc])=\min\{n, \{h_{j}+j-1 \ :  \ 1\leq j \leq n \}\}.
 \]
 \end{Theorem}
 \begin{proof}
By Theorem~\ref{thm:K[P]}, we have $K[\Pc]\iso K[\Pc^*]$. Therefore, it is enough to show that
\[
 \reg (K[\Pc^*])=\min\{n, \{h_{j}+j-1 \ :  \ 1\leq j \leq n \}\}.
 \]
Recall $G_{\Pc^*}$ is the bipartite graph associated to the vertices of $\Pc^*$. We may assume that
$V(G_{\Pc}^*)=\{x_{0},\ldots,x_{m}\} \sqcup \{y_{0},\ldots,y_{n}\}$. Then $\deg y_0=m+1\geq 2$ and $\deg x_{0}=n+1$. Hence, \cite[Proposition 5.7]{CN} gives
\begin{eqnarray*}
 \reg(K[G_{\Pc^*}])&=&\min \{n, \{\deg y_j +(j+1) -3 \ | \ 1 \leq j \leq n  \} \}=\\
 &=&\min \{n, \{\deg y_j +j-2 \ | \ 1 \leq j \leq n  \} \}
 \end{eqnarray*}
 We want to rewrite the formula above in terms of the horizontal projection of $\Pc^*$. According to Remark~\ref{degreeprojection}.(2), for any $1\leq j \leq n$ we have $h_{j}=\deg y_j -1$. Hence
 \[
\{ \deg y_j +j -2 \ | \ 1 \leq j \leq n \}=\{h_{j} +j-1 \ | \ 1 \leq j \leq n  \},
 \]
 and the assertion follows.
 \end{proof}

Let $\Pc$ be a Ferrer diagram with horizontal projections $(h_1,\ldots, h_{n})$. Then, by using a combinatorial argument, it is easy to see that for any $r \leq n$ the number of ways of placing $r$ rooks in non-attacking position in $\Pc$ is given by
\begin{equation}\label{r rooks}
\prod\limits_{i=1}^{r} (h_{r-i+1}-(i-1)).
\end{equation}
By using this fact we obtain
\begin{Theorem}\label{thm:rookreg}
Let $\Pc$ be an L-convex polyomino. Then
\[
\reg (K[\Pc])=r(\Pc).
\]
\end{Theorem}
\begin{proof}
From Lemma \ref{lem:rooks} we know that $r(\Pc)=r(\Pc^*)$ where $\Pc^*$ is the Ferrer diagram projected by $\Pc$. By Theorem~\ref{thm:K[P]}, it is enough to show that
\begin{equation}\label{eq:rook}
r(\Pc^*)=\min\{n, \{h_{j}+j-1 \ :  \ 1\leq j \leq n \}\},
\end{equation}
where $(h_1,\ldots, h_{n})$ are the horizontal projections of $\Pc^*$.
It follows from Equation \eqref{r rooks} that $r(\Pc^*)$ is the greatest integer $r \leq n$ such that each factor of $\prod\limits_{i=1}^{r} (h_{r-i+1}-(i-1))$ is positive. Hence, for any $i \in \{1,\ldots,r\}$ we must have $h_{r-i+1}-(i-1)>0.$
Fix an integer $i \in \{1,\ldots,r\}$. Then we see that
\[
h_{r-i+1}-(i-1)>0 \Leftrightarrow h_{r-i+1}-i+1+r-r>0 \Leftrightarrow r < h_{r-i+1}+(r-i)+1.
\]
Hence we conclude that $r \leq h_{r-i+1}+(r-i)$. Therefore,
\[
r(\Pc^*)=\max\{r: r \leq n \mbox{ and } r \leq \min\{h_{r-i+1}+(r-i) \ | \ 1\leq i \leq r\}\}
\]
\[=\min\{n, \{h_{j}+j-1 \ :  \ 1\leq j \leq n \}\}
\]
as requested.
\end{proof}

We observe that, by exchanging the role of rows and columns in $\Pc^*$, we obtain
\[
r(\Pc^*)=\min\{m, \{v_{j}+j-1 \ :  \ 1\leq j \leq m \}\}
\]
which is similar to \eqref{eq:rook}.

\section{On the Gorenstein property of $L$-convex polyominoes}
Let $\Pc$ be an $L$-convex polyomino. Assume that the unique rectangular subpolyomino of $\Pc$ having width $m$ has height $d \in \mathbb{N}$. Then for some positive integer $s$,
\[
H_{\Pc}=(h_1,\ldots, h_{s},m \ldots, m, h_{s+d+1},\ldots, h_{n})
\]
 with $h_{s},h_{s+d+1}<m$.
Let $\Pc_1$ be the collection of cells with $n-d$ rows satisfying the following property: $C_{ij}$ is a cell of $\Pc$ if and only if $C_{ij}$ is a cell of $\Pc_1$ for $1\leq  i \leq s$, and for $s+d+1 \leq  i \leq n$, $C_{i-d,j}$ is a cell of $\Pc_1$.

\begin{Lemma}
$\Pc_1$ is an $L$-convex polyomino.
\end{Lemma}
\begin{proof}
$\Pc_1$ could be seen as the polyomino $\Pc$ from which we remove the maximal rectangle $R$ having width $m$. Hence, each cell in $\Pc_1$ corresponds uniquely to a cell in $\Pc$. Let $C,D \in \Pc_1$. Then we consider the corresponding cells $C',D' \in \Pc$. We observe that neither $C'$ nor $D'$ is a cell of $R$. Since $\Pc$ is $L$-convex, there exists a path of cells $\Cc'$ in $\Pc$ connecting $C'$ and $D'$ with at most one change of direction. \\
If no cell of $\Cc'$ belongs to $R$, then $\Cc'$ determines a path of cells $\Cc$ of $\Pc_1$ with at most one change of direction connecting $C$ and $D$.\\
Otherwise, since neither $C'$ nor $D'$ are cells of $R$, the path $\Cc'$ crosses $R$ and the induced path $\Cc' \cap R$ has no change of direction. Therefore, the path $\Cc$ in $\Pc_1$, obtained by cutting off the induced path $\Cc' \cap R$ from $\Cc'$, is a path of cells with at most one change of direction connecting $C$ and $D$.
\end{proof}

If $\Pc_1\neq \emptyset$, we may again remove the unique rectangle of maximal width from $\Pc_1$ to obtain $\Pc_2$ in a similar way.  After a finite number of such steps, say $t$ steps,  we arrive at $\Pc_t$ which is a rectangle. Then   $\Pc_{t+1}=\emptyset$.  We set $\Pc_0=\Pc$, and call the sequence $\Pc_0,\Pc_1,\ldots, \Pc_t$ the {\em derived  sequence} of $L$-convex polyominoes of $\Pc$.
   \begin{figure}[H]
  \centering
  \begin{subfigure}{0.25\textwidth}
   \centering
    \resizebox{0.9\textwidth}{!}{
  \begin{tikzpicture}

\draw[thick] (-1,1) --  (-1,2);
\draw[thick] (0,1) --  (0,3);
\draw[thick] (1,0) --  (1,5);
\draw[thick] (2,0) --  (2,5);
\draw[thick] (3,0) --  (3,4);
\draw[thick] (4,1) --  (4,2);

\draw[thick] (-1,1) --  (4,1);
\draw[thick] (-1,2) --  (4,2);
\draw[thick] (0,3) --  (3,3);
\draw[thick] (1,4) --  (3,4);
\draw[thick] (1,5) --  (2,5);
\draw[thick] (1,0) --  (3,0);

\fill[fill=gray, fill opacity=0.2] (1,2) -- (2,2)-- (2,3) -- (1,3);
\fill[fill=gray, fill opacity=0.2] (0,1) -- (0,3)-- (1,3) -- (1,1);
\fill[fill=gray, fill opacity=0.2] (2,0) -- (2,4)-- (3,4) -- (3,0);
\fill[fill=gray, fill opacity=0.2] (1,3) -- (1,5)-- (2,5) -- (2,3);
\fill[fill=gray, fill opacity=0.2] (3,1) -- (3,2)-- (4,2) -- (4,1);
\fill[fill=gray, fill opacity=0.2] (1,1) -- (1,2)-- (2,2) -- (2,1);
\fill[fill=gray, fill opacity=0.2] (1,0) -- (2,0)-- (2,1) -- (1,1);
\fill[fill=gray, fill opacity=0.2] (-1,1) -- (0,1)-- (0,2) -- (-1,2);
   \end{tikzpicture}}

   \end{subfigure}%
  \begin{subfigure}{0.25\textwidth}
   \centering
    \resizebox{0.6\textwidth}{!}{
  \begin{tikzpicture}

\draw[thick] (0,1) --  (0,2);
\draw[thick] (1,0) --  (1,4);
\draw[thick] (2,0) --  (2,4);
\draw[thick] (3,0) --  (3,3);

\draw[thick] (0,1) --  (3,1);
\draw[thick] (0,2) --  (3,2);
\draw[thick] (1,3) --  (3,3);
\draw[thick] (1,4) --  (2,4);
\draw[thick] (1,0) --  (3,0);

\fill[fill=gray, fill opacity=0.2] (1,0) -- (3,0)-- (3,3) -- (1,3);
\fill[fill=gray, fill opacity=0.2] (1,3) -- (1,4)-- (2,4) -- (2,3);
\fill[fill=gray, fill opacity=0.2] (0,1) -- (1,1)-- (1,2) -- (0,2);

   \end{tikzpicture}}
   \end{subfigure}%
     \begin{subfigure}{0.25\textwidth}
   \centering
    \resizebox{0.4\textwidth}{!}{
  \begin{tikzpicture}

\draw[thick] (1,0) --  (1,3);
\draw[thick] (2,0) --  (2,3);
\draw[thick] (3,0) --  (3,2);

\draw[thick] (1,1) --  (3,1);
\draw[thick] (1,2) --  (3,2);
\draw[thick] (1,3) --  (2,3);
\draw[thick] (1,0) --  (3,0);

\fill[fill=gray, fill opacity=0.2] (1,0) -- (3,0)-- (3,2) -- (1,2);
\fill[fill=gray, fill opacity=0.2] (1,2) -- (1,3)-- (2,3) -- (2,2);

   \end{tikzpicture}}
   \end{subfigure}%
   \begin{subfigure}{0.25\textwidth}
   \centering
    \resizebox{0.2\textwidth}{!}{
  \begin{tikzpicture}

\draw[thick] (0,0) --  (1,0);
\draw[thick] (0,0) --  (0,1);
\draw[thick] (1,0) --  (1,1);
\draw[thick] (0,1) -- (1,1);

\fill[fill=gray, fill opacity=0.2] (0,0) -- (0,1)-- (1,1) -- (1,0);

   \end{tikzpicture}}
   \end{subfigure}
   \caption{The derived sequence of $L$-convex polyominoes $\Pc_0=\Pc,\Pc_1,\Pc_2,\Pc_3$.}
   \end{figure}

\begin{Lemma}\label{lem:chocolate}
Let $\Pc$ be an $L$-convex polyomino, let $\Pc_0, \Pc_1, \ldots, \Pc_{t}$ be the derived sequence of $L$-convex polyominoes of $\Pc$. Let $\Pc^*$ be the Ferrer diagram projected by $\Pc$ and let $(\Pc^*)_0$, $(\Pc^*)_1$, $\ldots, (\Pc^*)_{t'}$. Then $t'=t$ and for any $0\leq k\leq t$ the polyomino $(\Pc^*)_k$ is the Ferrer diagram projected by $\Pc_{k}$. In other words, for all $k$ $(\Pc^*)_k=(\Pc_k)^*$.
\end{Lemma}
\begin{proof}
For $k=0$, the assertion is trivial.
We show that $(\Pc^*)_1$ is the Ferrer diagram projected by $\Pc_1$. For this aim, assume that the unique rectangular subpolyomino of $\Pc$ having width $m$ has height $d \in \mathbb{N}$. Let
\[
H_{\Pc}=(h_1,\ldots, h_{s},m \ldots, m, h_{s+d+1},\ldots, h_{n})
\]
 with $h_{s},h_{s+d+1}<m$ and let
 \[
 V_{\Pc}=(d,d,\ldots,d,v_{r+1},\ldots, v_{r+l},d,\ldots,d)
 \]
with  $v_{r+1},v_{r+l}>d$.

From Proposition \ref{prop:Ferrer} it follows that $\Pc^*$ has a maximal rectangle $R^*$ of width $m$ and height $d$
and
\[
H_{\Pc^*}=(m \ldots, m, h^*_{1},\ldots, h^*_{n-d})
\]
with $m>h^*_{1}\geq  \cdots \geq h^*_{n-d}$ and
 \[
 V_{\Pc^*}=(v^*_{1},\ldots, v^*_{l},d,\ldots,d).
 \]
 with $v^*_{1}\geq \cdots \geq v^*_{l}>d$. Hence the $L$-convex polyomino $(\Pc^*)_1$ is uniquely determined by the projections
 \[
H_{(\Pc^*)_1}=(h^*_{1},\ldots, h^*_{n-d}) \mbox{ and } V_{(\Pc^*)_1}=(v^*_{1}-d,\ldots, v^*_{l}-d).
\]
On the other hand, $\Pc_{1}$ is the $L$-convex polyomino uniquely determined by the projections $H_{\Pc_1}=(h_1,\ldots, h_{s}, h_{s+d+1},\ldots, h_{n})$ and, $V_{\Pc_1}=(v_{r+1}-d,v_{r+2}-d, \ldots, v_{r+l}-d)$. By reordering the two vectors in a decreasing order, we obtain the Ferrer diagram projected by $\Pc_1$ which coincides with $(\Pc^*)_1$. This proves the assertion for $k=1$.
By inductively applying the above argument, the assertion follows for all $k$.
\end{proof}
\begin{Theorem}
\label{gorenstein}
Let $\Pc$ be an $L$-convex polyomino and let $\Pc_0, \Pc_1, \ldots, \Pc_{t}$ be the derived sequence of $L$-convex polyominoes of $\Pc$. Then following conditions are equivalent:

\begin{enumerate}
\item[{\em (a)}] $\Pc$ is Gorenstein.
\item[{\em (b)}]For  $0\leq k\leq t$, the bounding box  of $\Pc_k$ is a square.
\end{enumerate}
\end{Theorem}
\begin{proof}
By Theorem~\ref{thm:K[P]}, we have $K[\Pc]\iso K[\Pc^*]$, where 
$\Pc^*$ is the Ferrer diagram projected by $\Pc$. Therefore, $K[\Pc]$ is Gorenstein if and only if $K[\Pc^*]$ is Gorenstein. Note that $\Pc^*$ can be viewed as a stack polyomino. Hence it follows from  \cite[Corollary 4.12]{Q} that $K[\Pc^*]$ is Gorenstein if and only if the bounding box  of $(\Pc^*)_k$ is a square for all $0\leq k\leq t$. By Lemma~\ref{lem:chocolate}, this is the case if and only if the bounding box of $\Pc_k$ is a square for all $0\leq k\leq t$.
\end{proof}

The following numerical criteria for the Gorensteinness of $\Pc$ are an immediate consequence of Theorem~\ref{gorenstein}.

\begin{Corollary}
\label{againgorenstein}
Let $\Pc$ be an $L$-convex polyomino with the vector $H_{\Pc}=(h_1, h_2, \ldots, h_{n})$ of  horizontal projections of $\Pc$ and the vector  $V_{\Pc}=(v_1, v_2, \ldots, v_{m})$ of  vertical projections of $\Pc$. We set
\[
\{h_1,h_2,\ldots,h_n\}=\{g_1<g_2<\cdots <g_r\} \text{ and } \{v_1,v_2,\ldots,v_m\}=\{w_1<w_2<\cdots <w_s\},
\]
and let
\[
a_i=|\{h_j\:\; h_j= g_i\}| \text{ for } i=1,\ldots,r,
\text{ and }
b_i=|\{v_j\:\; v_j= w_i\} |\text{ for } i=1,\ldots,s.
\]
Then the following conditions are equivalent:
\begin{enumerate}
\item[{\em (a)}] $\Pc$ is Gorenstein.
\item[{\em (b)}] $g_\ell=\sum_{i=1}^\ell a_i$ for $\ell=1,\ldots,r$.
\item[{\em (c)}] $w_\ell=\sum_{i=1}^\ell b_i$ for $\ell=1,\ldots,s$.
\end{enumerate}
\end{Corollary}

\begin{Theorem}
\label{late}
Let $\Pc$ be $L$-convex polyominoes such that $K[\Pc]$ is not Gorenstein. Then following are equivalent:
\begin{enumerate}
\item[{\em (a)}]  $K[\Pc]$ is Gorenstein on the punctured spectrum.
\item[{\em (b)}] $\Pc$ is not a square and $K[\Pc]$ has an isolated singularity
\item[{\em (c)}] $\Pc$ is rectangle, but not a square.
\end{enumerate}
\end{Theorem}

Before we start the proof of the theorem, we note that if $\Pc$ is a Ferrer diagram, then $K[\Pc]$ may be viewed as a Hibi ring. Recall for a  given finite poset $Q= \{v_1,\ldots,v_n\}$ and a field $K$, the \textit{Hibi ring} over the field $K$ associated to $Q$, which we denote by $K[Q] \subset K[y,x_1,\dots,x_n]$, is defined as follows. The $K$-algebra $K[Q]$ is generated by the monomials $yx_I := y\prod_{v_i \in I} x_i$ for every $I \in \mathcal{I}(Q)$, that is
\[
K[Q] := K[yx_I | I \in \mathcal{I}(Q)].
\]
The algebra $K[Q]$ is standard graded if we set $\deg(yx_I)=1$ for all $I \in \mathcal{I}(Q)$.
Here $\mathcal{I}(Q)$ is the set of poset ideals of $Q$. The poset ideals of $Q$ are just the subset $I\subset Q$ with the property that if $p\in Q$ and $q\leq p$, then $q\in Q$. 

Let $\Pc$ be a Ferrer diagram with maximal horizontal edge intervals $\{H_{0},\ldots,H_{n}\}$, numbered increasingly from the bottom to the top, and maximal vertical edge intervals $\{V_{0},\ldots,V_{m}\}$, numbered increasingly from the left to the right. We let  $Q$ be  the poset on the  set  $\{H_{1},\ldots,H_{n},V_{1},\ldots,V_{m}\}$  consisting of two chains $H_1<\ldots < H_{n}$ and $V_1<\ldots < V_m$ and the covering relations $H_i <V_{j}$, if $H_{i}$ intersects $V_{j}$ in a way such that there is no $0\leq i'<i$ for which $H_{i'}$ intersects $V_{j}$, and $j$ is the smallest integer with this property.

\begin{Lemma}
The standard graded $K$-algebras $K[Q]$ and $K[\Pc]$ are isomorphic.
\end{Lemma}
\begin{proof}
We may assume that the interval $[(0,0),(m,n)]$ is  the bounding box of the Ferrer diagram $\Pc$. For any two vertices $a=(i,j)$ and $b=(k,l)$ of $\Pc$  we define the  meet $a\wedge  b=(\min\{i,k\}, \min\{j,l\})$ and the join $a\vee b=(\max\{i,k\}, \max\{j,l\})$. With this  operations of meet and join, $\Pc$ is a distributive lattice. An element $c$ of this lattice is called join-irreducible,  if $c\neq (0,0)$ and whenever $a\wedge b=c$, then $a=c$ or $b=c$. By Birkhoff's fundamental structure theorem \cite{Bi}, any finite distributive lattice is the ideal lattice of the poset of its joint irreducible elements. In our case the poset of join irreducible elements of $\Pc$ is exactly the poset $Q$ described above. Thus  the elements $a\in \Pc$ are in bijection with   the poset ideals of $Q$. In fact, the poset ideal $I_a\in \mathcal{I}(Q)$  corresponding ot $a$ is the set of join irreducible elements $q\in Q$ with $q\leq a$. Thus we have a surjective $K$-algebra homomorphism $$\varphi \: S=K[x_a\:\; a\in \Pc]\to K[Q]= K[yx_{I_a} \:\;  I_a \in \mathcal{I}(Q)]$$. 
As shown by Hibi \cite{Hi} (see also \cite[Theorem 10.1.3]{HH}), $\Ker(\varphi)$ is generated by the relations $x_ax_b-x_{a\wedge b}x_{a\vee b}$. This shows that $\Ker(\varphi)= I_{\Pc}$, as desired.
\end{proof}

 \begin{figure}[H]
   \centering
   \begin{subfigure}{0.5\textwidth}
   \centering
  \resizebox{0.8\textwidth}{!}{
  \begin{tikzpicture}

\draw (0,5)--(5,5);
\draw(0,4)-- (5,4);
\draw(0,3)-- (3,3);
\draw(0,2)-- (3,2);
\draw(0,1)-- (2,1);
\draw(0,0)-- (1,0);

\draw (0,0)--(0,5);
\draw (1,0)--(1,5);
\draw (2,1)--(2,5);
\draw (3,2)--(3,5);
\draw (4,4)-- (4,5);
\draw (5,4)-- (5,5);

%\fill[fill=gray, fill opacity=0.2] (1,2) -- (2,2)-- (2,3) -- (1,3);

\filldraw (1,0)  circle (1.5*\rad) node[anchor=north] {$V_1$}; 
\filldraw (2,1)  circle (1.5*\rad) node[anchor=north] {$V_2$}; 
\filldraw (3,2)  circle (1.5*\rad) node[anchor=north] {$V_3$}; 
\filldraw (4,4)  circle (1.5*\rad) node[anchor=north] {$V_4$}; 
\filldraw (5,4)  circle (1.5*\rad) node[anchor=north] {$V_5$}; 
\filldraw (0,1)  circle (1.5*\rad) node[anchor=east] {$H_1$}; 
\filldraw (0,2)  circle (1.5*\rad) node[anchor=east] {$H_2$}; 
\filldraw (0,3)  circle (1.5*\rad) node[anchor=east] {$H_3$}; 
\filldraw (0,4)  circle (1.5*\rad) node[anchor=east] {$H_4$}; 
\filldraw (0,5)  circle (1.5*\rad) node[anchor=east] {$H_5$}; 

   \end{tikzpicture}}
   \caption{Ferrer diagram}
\end{subfigure}%
 \begin{subfigure}{0.5\textwidth}
 \centering
  \resizebox{0.5\textwidth}{!}{
  \begin{tikzpicture}
 
\draw (0,1)--(0,5);
\draw (2,2)--(2,6);

\draw(0,1)--(2,3);
\draw(0,2)--(2,4);
\draw(0,4)--(2,5);

\filldraw (0,1)  circle (1.5*\rad) node[anchor=east] {$H_1$}; 
\filldraw (0,2)  circle (1.5*\rad) node[anchor=east] {$H_2$}; 
\filldraw (0,3)  circle (1.5*\rad) node[anchor=east] {$H_3$}; 
\filldraw (0,4)  circle (1.5*\rad) node[anchor=east] {$H_4$}; 
\filldraw (0,5)  circle (1.5*\rad) node[anchor=east] {$H_5$};
\filldraw (2,2)  circle (1.5*\rad) node[anchor=west] {$V_1$}; 
\filldraw (2,3)  circle (1.5*\rad) node[anchor=west] {$V_2$}; 
\filldraw (2,4)  circle (1.5*\rad) node[anchor=west] {$V_3$}; 
\filldraw (2,5)  circle (1.5*\rad) node[anchor=west] {$V_4$}; 
\filldraw (2,6)  circle (1.5*\rad) node[anchor=west] {$V_5$}; 

   \end{tikzpicture}}\ \\
   \caption{Poset of join-irreducible elements}
\end{subfigure}
\end{figure}

Let $Q$ be a poset. The Hasse diagram of $Q$, viewed as a graph, decomposes into connected components. The corresponding posets $Q_1,\ldots,Q_{r}$ are called the \emph{connected components} of $Q$.

Now for  the proof of Theorem~\ref{late} will use the following results

\begin{Theorem}
\label{fresh}
Let $Q$ be a finite poset and let $Q_1,\ldots Q_r$ be the connected components of $Q$. 
\begin{enumerate}
\item[{\em (a)}] \emph{(}\cite[Corollary]{Hi}\emph{)} $K[Q]$ is Gorenstein if and only if $Q$ is pure (i.e.\ ~ all maximal chains in $Q$ have the same length).
\item[{\em (b)}] \emph{(}\cite[Corollary~3.5]{HMP}\emph{)} $K[Q]$ is Gorenstein on the punctured spectrum if and only if each $Q_i$ is pure. 
\end{enumerate}
\end{Theorem}

\begin{proof}[Proof of Theorem~\ref{late}]
Since $K[\Pc]\iso K[\Pc^*]$ and since $\Pc$ is a rectangle if and only if $\Pc^*$ is a rectangle, we may assume that $\Pc$ is a Ferrer diagram. 

Let $Q$ be the poset such that $K[Q]\iso K[\Pc]$, and assume that $K[Q]$ is Gorenstein on the punctured spectrum. Then each component of $Q$ is pure, by Theorem~\ref{fresh}(b). Since we assume that $K[Q]$ is not Gorenstein, Theorem~\ref{fresh}(a)  implies that $Q$ is not connected. It follows from the description of $Q$ in terms of its Ferrer diagram $\Pc$ that $\Pc$ has no inner corner. In other words, $\Pc$ is  a rectangle. By Theorem~\ref{gorenstein} it cannot be  a square. This yields (a)\implies (b). The implication (c) \implies (b) follows from \cite[Theorem 2.6]{BV}, and (b)\implies (a) is trivial.
\end{proof}

\section{The Cohen--Macaulay type of $L$-convex polyominoes.}
In this section, we give a general formula for the Cohen--Macaulay type of the coordinate ring of an $L$-convex polyomino. To illustrate our result, we first consider the special case of an $L$-convex polyomino with just two maximal rectangles.

\begin{Proposition}\label{prop:CM2R}
Let $\Pc$ be an $L$-convex polyomino whose maximal rectangles are $R_1$ having size $m \times s$ and $R_2$ having size $t \times n$ with $s<n$ and $t<m$. 
Let $r=\max\{n,m,n+m-(s+t)\}$. Then
\[
\mathrm{type}(K[\Pc])=\begin{cases} \sum\limits_{i=m-t}^{m-(n-s)}\binom{i}{s}\binom{m-i-1}{n-s-1} &\mbox{if } r=m \\ \sum\limits_{i=m-t}^{s}\binom{i-1}{m-t-1}\binom{n-i}{t}  &\mbox{if } r=n \\ \binom{n-s}{t}\binom{m-t}{s}  &\mbox{if } r=n+m-(s+t) \end{cases}
\]
\end{Proposition}
\begin{proof}
Let $\Pc^*$ be the Ferrer diagram projected by $\Pc$ and let $Q$ be the poset of the join-irreducible elements associated to $\Pc^*$.
It consists of the two chains $V_{1}<\cdots <V_{m}$ and $H_{1}<\cdots <H_{n}$, and the cover relation $H_{n-s} < V_{t+1}$.
We have $|Q|=m+n$, and $r=\rank Q +1$. We compute the number of minimal generators of the canonical module $\omega_{K[\Pc^*]}$.
For this purpose, let $\widehat{Q}$ be the poset obtained from $Q$ by adding the elements $-\infty $ and $\infty$ with $\infty > p$ and $-\infty < p$ for all $p \in Q$, and let $\MT(\widehat{Q})$ be the set of integer valued functions $\nu: \widehat{Q}\to \mathbb{Z}_{\geq 0}$ with $\nu(\infty) = 0$ and $\nu(p) < \nu(q)$ for all $p > q$. By using a result of Stanley \cite{St}, Hibi shows in \cite[(3.3)]{Hi} that the monomials of the form
\[
y^{\nu(-\infty)} \prod\limits_{p \in Q} x_{p}^{\nu(p)}
\]
for $\nu \in \MT(\widehat{Q})$ form a $K$-basis for $\omega_{K[\Pc^*]}$. 	
By using \cite[Corollary 2.4]{Mi}, we have that the number of generators of $\omega_{K[\Pc^*]}$ is the number of minimal maps $\nu \in \MT(\widehat{Q})$ with respect to the order given in \cite[Page 5]{Mi}. In fact, $\nu \leq \mu$ for $\nu,\mu \in T(\widehat{Q})$ if $\mu-\nu$ is decreasing. 
We observe that the minimal maps $\nu$ necessarily assign the numbers $1,\ldots,r$ to the vertices of a maximal chain of $Q$ in reversed order, hence we have to find the possible values for the remaining $|Q|-r=m+n-r$ elements, depending on $r$. We distinguish three cases:
\begin{itemize}
\item[ (a)] $r=m$;
\item[(b)] $r=n$;
\item[(c)] $r=(n-s)+(m-t)$
\end{itemize}
In the case (a), the maximal chain is  $V_1<\cdots<V_{m}$. Hence we must take $\nu(V_{m-i+1})=i$ for  $i \in \{1,\ldots, m\}$. We have to determine how many vectors $(a_1,\ldots,a_{n})$ with integers entries $0<a_{1}<\cdots<a_{n}$ satisfy $m-t<a_{s+1}<r-(n-s)=m-(n-s)+2$, where the left inequality follows from the cover relation, while the right inequality follows from the fact that $a_{s+2}<\cdots<a_{n}<m+1$ are determined. Therefore, fixed $i=a_{s+1}$, there are $\binom{i-1}{s}$ ways to choose the values $a_{1},\ldots,a_{s}$ in the range $\{1,\ldots,i-1\}$. Moreover, there are $\binom{m-i}{n-s-1}$ ways to choose $a_{s+2},\ldots, a_{n}$ in the range $\{i+1,\ldots, m\}$. Hence we conclude
\[
\mathrm{type}(K[\Pc])=\sum\limits_{i=m-t+1}^{m-(n-s)+1}\binom{i-1}{s}\binom{m-i}{n-s-1}=\sum\limits_{i=m-t}^{m-(n-s)}\binom{i}{s}\binom{m-i-1}{n-s-1}.
\]
In the case (b), we assign to each element of the chain  $H_1,\ldots,H_{n}$ a number in $\{1,\ldots, n\}$ in strictly decreasing order. We have to determine how many vectors $(b_1,\ldots,b_{m})$ with integers entries $0<b_{1}<\cdots<b_{m}$ satisfy $m-t-1<b_{m-t}<s+1$, where the left inequality follows from the fact that $0<b_{1}<\cdots<b_{m-t-1}$, while the rightmost inequality follows from the cover relation. Therefore, fixed $i=b_{m-t}$, there are $\binom{i-1}{m-t-1}$ ways to choose the values $b_{1},\ldots,b_{m-t-1}$ in the range $\{1,\ldots,i-1\}$. Moreover, there are $\binom{n-i}{t}$ ways to choose $b_{m-t+1},\ldots, b_{m}$ in the range $\{i+1,\ldots, n\}$. Hence we conclude
\[
\mathrm{type}(K[\Pc])=\sum\limits_{i=m-t}^{s}\binom{i-1}{m-t-1}\binom{n-i}{t}.
\]

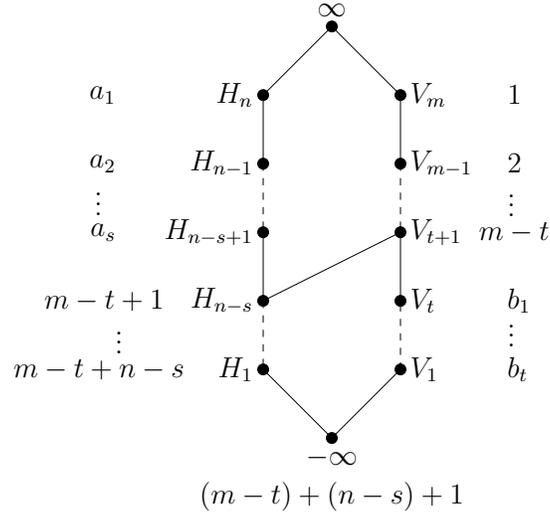
\begin{figure}[H]
 \centering
  \resizebox{0.5\textwidth}{!}{
  \begin{tikzpicture}
 
\draw (0,4)--(0,5);
\draw[dashed] (0,4)--(0,3);
\draw (0,2)--(0,3);
\draw[dashed] (0,2)--(0,1);
\draw (2,4)--(2,5);
\draw[dashed] (2,4)--(2,3);
\draw (2,2)--(2,3);
\draw[dashed] (2,1)--(2,2);

\draw (1,6)--(0,5);
\draw (1,6)--(2,5);
\draw (1,0)--(0,1);
\draw (1,0)--(2,1);
\draw(2,3)--(0,2);

\filldraw (1,6) circle (1.5*\rad) node[anchor=south] {$\infty$}; 
\filldraw (1,0) circle (1.5*\rad) node[anchor=north] {$-\infty$}; 
\filldraw (0,1)  circle (1.5*\rad) node[anchor=east] {$H_1$}; 
\filldraw (0,2)  circle (1.5*\rad) node[anchor=east] {$H_{n-s}$}; 
\filldraw (0,3)  circle (1.5*\rad) node[anchor=east] {$H_{n-s+1}$};
\filldraw (0,4)  circle (1.5*\rad) node[anchor=east] {$H_{n-1}$}; 
\filldraw (0,5)  circle (1.5*\rad) node[anchor=east] {$H_n$};
\filldraw (2,1)  circle (1.5*\rad) node[anchor=west] {$V_1$}; 
\filldraw (2,2)  circle (1.5*\rad) node[anchor=west] {$V_t$}; 
\filldraw (2,3)  circle (1.5*\rad) node[anchor=west] {$V_{t+1}$}; 
\filldraw (2,4)  circle (1.5*\rad) node[anchor=west] {$V_{m-1}$}; 
\filldraw (2,5)  circle (1.5*\rad) node[anchor=west] {$V_m $}; 

\filldraw (1,-0.5) circle (0*\rad) node[anchor=north] {$(m-t)+(n-s)+1$};
\filldraw (3,5) circle (0*\rad) node[anchor=west] {$\ \ \ 1$};
\filldraw (3,4) circle (0*\rad) node[anchor=west] {$\ \ \ 2$};
\filldraw (3,3.5) circle (0*\rad) node[anchor=west] {$\ \ \ \vdots$};
\filldraw (3,3) circle (0*\rad) node[anchor=west] {$m-t$};
\filldraw (-1.3,2)  circle (0*\rad) node[anchor=east] {$m-t+1$}; 
\filldraw (-1.9,1.5)  circle (0*\rad) node[anchor=east] {$\vdots$}; 
\filldraw (-1,1)  circle (0*\rad) node[anchor=east] {$m-t+n-s$}; 

\filldraw (-2,3)  circle (0*\rad) node[anchor=east] {$a_s$};
\filldraw (-2.2,3.52) circle (0*\rad) node[anchor=east] {$\vdots$};
\filldraw (-2,4)  circle (0*\rad) node[anchor=east] {$a_2$}; 
\filldraw (-2,5)  circle (0*\rad) node[anchor=east] {$a_1$};
\filldraw (3,2) circle (0*\rad) node[anchor=west] {$\ \ \ b_1$};
\filldraw (3,1.6) circle (0*\rad) node[anchor=west] {$\ \ \ \vdots$};
\filldraw (3,1) circle (0*\rad) node[anchor=west] {$\ \ \ b_t$};

   \end{tikzpicture}}\ \\
\caption{We count the number of minimal maps assigning $1~<~\cdots~ <m-t+n-s $ to $V_{m}>\cdots >V_{t+1}> H_{n-s} >\cdots >H_1$.}\label{fig:pos}
\end{figure}

In the case (c), we assign to each element of the chain  $H_1,\ldots,H_{n-s},V_{t+1},\ldots, V_{m}$ a number in $\{1,\ldots, (m+n)-(s+t)\}$ in strictly decreasing order. We have to determine how many vectors $(a_1,\ldots,a_{s},b_{1},\ldots b_{t})$ with integers entries $0<a_{1}<\cdots<a_{s}<b_{1}<\cdots<b_{t}$ satisfy $a_{s}<m-t+1$ and $b_{1}>m-t$ (see Figure \ref{fig:pos}). Therefore, there are $\binom{m-t}{s}$ ways to choose the values $a_{1},\ldots,a_{s}$ in the range $\{1,\ldots,m-t\}$ and there are $\binom{n-s}{t}$ ways to choose $b_{1},\ldots, b_{t}$ in the range $\{m-t+1,\ldots, m-t+n-s\}$. Hence in this case, we conclude
\[
\mathrm{type}(K[\Pc])=\binom{n-s}{t}\binom{m-t}{s}.
\]
\end{proof}

Now we consider the general case.

\begin{Theorem}\label{thm:type}
Let $\Pc$ be an $L$-convex polyomino whose maximal rectangles are $\{R_i\}_{i=1,\ldots,t}$. For $i=1,\ldots ,t$, let $c_{i} \times d_{i}$ be the size of $R_i$ with $d_{1}=n$ and $c_{t}=m$ and $c_{i}<c_{j}$ and $d_{i}>d_{j}$ for $i<j$. Let $r=\max\{n,m,\{n+m-(c_{i}+d_{i+1})\}_{i=1,\ldots,t-1}\}$. Then
\[
\mathrm{type}(K[\Pc])=\begin{cases} A &\mbox{if } r=m \\ B &\mbox{if } r=n \\ A_{h}B_{h} &\mbox{if } r=n+m-(c_{h}+d_{h+1}) \end{cases}
\]
where
\[
A=\sum\limits_{i_1,\ldots,i_{t-1}} \binom{i_1-1}{d_{t}} \prod\limits_{k=2}^{t-1} \binom{i_k-i_{k-1}-1}{d_{t-k+1}-d_{t-k+2}-1} \binom{m-i_{t-1}}{n-d_2-1}
\]
with 
\[
m-c_{t-j}+1\leq i_{j} \leq m-(n-d_{t+1-j})+1  \ \emph{  for  } \  1\leq j \leq t-1,
\]
and
\[
B=\sum\limits_{i_1,\ldots,i_{t-1}} \binom{i_1-1}{m-c_{t-1}-1} \prod\limits_{k=2}^{t-1} \binom{i_k-i_{k-1}-1}{c_{t-k+1}-c_{t-k}-1} \binom{m-i_{t-1}}{c_1}
\]
with $m-c_{t-1}\leq i_{1} \leq d_{t}$ and
\[
i_{j-1}+c_{t-j+1}-c_{t-j}\leq i_{j} \leq d_{t-j+1}  \ \emph{  for  } \  2\leq j \leq t-1.
\]
Moreover,
\begin{eqnarray*}
&A_{h}& = \sum\limits_{i_{1},\ldots, i_{t-h-1}} \binom{i_1-1}{d_{t}}\prod\limits_{k=2}^{t-h-1} \binom{i_{k}-i_{k-1}-1}{d_{t-k+1}-d_{t-k+2}-1} \binom{ m-c_{h}-i_{t-h-1}}{d_{h+1}-d_{h+2}-1}
\end{eqnarray*}
with
\[
m-c_{t-j}+1\leq i_{j} \leq m-c_{h}-(d_{h+1}-d_{t-j+1})+1  \ \emph{  for  } \  1\leq j \leq t-h-1.
\]
for $h=1,\ldots,t-2$ and $A_{t-1}=\binom{m-c_{t-1}}{d_t}$, and
\begin{eqnarray*}
B_{h}&=&\sum\limits_{i_1,\ldots,i_{h-1}} \binom{i_1-1}{c_{h}-c_{h-1}-1} \prod\limits_{k=2}^{h-1}\binom{i_{k}-i_{k-1}-1}{c_{h-k+1}-c_{h-k}-1} \\ &&\binom{m-c_{h}+n-d_{h+1}-i_{h-1}}{c_1}
\end{eqnarray*}
with $m-c_{h-1}\leq i_{1} \leq m-c_{h}+(d_{h}-d_{h+1})$ and
\[
i_{j-1}+(c_{h-j+1}-c_{h-j})\leq i_{j} \leq m-c_{h}+(d_{h-j+1}-d_{h+1})  \ \emph{  for  } \  2\leq j \leq h-1, 
\]
for $h=2,\ldots,t-1$ and $B_{1}=\binom{n-d_{2}}{c_1}$.
\end{Theorem}
\begin{proof}
Firstly observe that in the general case the cover relations are $H_{n-d_{i+1}}~<~V_{c_{i}+1}$ for $i=1,\ldots ,t-1$.
We just generalize the ideas of Proposition \ref{prop:CM2R}. We distinguish three cases:
\begin{itemize}
\item[(a)] $r=m$;
\item[(b)] $r=n$;
\item[(c)] $r=(n-d_{h+1})+(m-{c_{h}})$ for some $k=1,\ldots,t-1$.
\end{itemize}
In the case (a), we assign to each element of the chain  $V_1,\cdots,V_{m}$ a number in $\{1,\ldots, m\}$ in decreasing order. We have to determine how many vectors $(a_1,\ldots,a_{n})$ with integers entries $0<a_{1}<\cdots<a_{n}<m+1$ satisfy 
\[
m-c_{t-k} < a_{d_{t-k+1}+1} < m-(n-d_{t-k+1})+2 \ \ \ \mbox{for } k=1,\ldots,t-1,
\]
where the left inequality follows from the cover relations, while the right inequality follows from the fact that $a_{d_{t-k+1}+2}<a_{d_{t-k+1}+3}<\cdots<a_{n}<m+1$ . Therefore, fixed $i_1=a_{d_{t}+1}$  there are $\binom{i_1-1}{d_{t}}$ ways to choose the values $a_{1},\ldots,a_{d_{t}}$ in the range $\{1,\ldots,i_1-1\}$. Moreover, for $2\leq k \leq t-1$ and fixed $i_{k}=a_{d_{t-k+1}+1}$, there are $\binom{i_{k}-i_{k-1}-1}{d_{t-k+1}-d_{t-k+2}-1}$ ways to choose the values $a_{d_{t+k+2}+2},\ldots,a_{d_{t+k+1}}$ in the range $\{i_{k-1}+1,\ldots,i_{k}-1\}$. Finally, there are $\binom{m-i_{t-1}}{n-d_{2}-1}$ ways to choose $a_{d_{2}+2},\ldots, a_{n}$ in the range $\{i_{t-1}+1,\ldots, m\}$. Hence in this case, we conclude that $\mathrm{type}(K[\Pc])$ is $A$.\\
In the case (b), we assign to each element of the chain  $H_1,\ldots,H_{n}$ a number in $\{1,\ldots, n\}$ in decreasing order. We have to determine how many vectors $(b_1,\ldots,b_{m})$ with integers entries $0<b_{1}<\cdots<b_{m}$ satisfy 
\[
m-c_{t-1}-1 < b_{m-c_{t-1}} < d_{t}+1
\]
and
\[
b_{m-c_{t-k+1}}+(c_{t-k+1}-c_{t-k})-1 < b_{m-c_{t-k}} < d_{t-k+1}+1 \mbox{   for } k=2,\ldots,t-1,
\]
where the left inequalities follow from the fact that $b_{m-c_{t-k+1}+1}<\cdots<b_{m-c_{t-k}-1}$, while the right inequalities follow from the cover relations.
Therefore for fixed $i_1=b_{m-c_{t-1}}$  there are $\binom{i_1-1}{m-c_{t-1}-1}$ ways to choose the values $b_{1},\ldots,b_{m-c_{t-1}-1}$ in the range $\{1,\ldots,i_1-1\}$. Moreover, for $2\leq k \leq t-1$ and fixed $i_{k}= b_{m-c_{t-k}}$, there are $\binom{i_{k}-i_{k-1}-1}{c_{t-k+1}-c_{t-k}-1}$ ways to choose the values $b_{m-c_{t-k+1}+1},\ldots,b_{m-c_{t-k}-1}$ in the range $\{i_{k-1}+1,\ldots,i_{k}-1\}$. Finally, there are $\binom{n-i_{t-1}}{c_1}$ ways to choose $b_{m-c_1+1},\ldots, b_{m}$ in the range $\{i_{t-1}+1,\ldots, n\}$. Hence in this case, we conclude that $\mathrm{type}(K[\Pc])$ is $B$.\\
In the case (c), fix $h \in 1,\ldots,t-1$ we assign to each element of the chain  $H_1,\ldots$, $H_{n-d_{h+1}}$, $V_{c_{h}+1},\ldots, V_{m}$ a number in $\{1,\ldots, (m+n)-(c_{h}+d_{h+1})\}$ in decreasing order.  Let $\widetilde{m}=m-c_{h}$, $\widetilde{n}=n-d_{h+1}$.
We have to determine how many vectors $(a_1,\ldots,a_{d_{h+1}},b_{1},\ldots b_{c_{h}})$  with integers entries $0<a_{1}<\cdots<a_{d_{h+1}}<b_{1}<\cdots<b_{c_{h}}$ satisfy
\[
m-c_{t-k} < a_{d_{t-k+1}+1} < \widetilde{m}-(d_{h+1}-d_{t-k+1})+2 \mbox{\ \ \ for } k=1,\ldots,t-h-1
\]
\[
m-c_{h-1}-1<b_{c_{h}-c_{h-1}}<\widetilde{m}+(d_{h}-d_{h+1})+1,
\]
\[
b_{c_{h}-c_{h-k+1}}+(c_{h-k+1}-c_{h-k})-1< b_{c_h-c_{h-k}} < \widetilde{m}+(d_{h-k+1}-d_{h+1})+1 \mbox{   for } k=2,\ldots,h-1.
\]
For fixed $i_1=a_{d_{t}+1}$ there are $\binom{i_1-1}{d_{t}}$ ways to choose the values $a_{1},\ldots,a_{d_{t}}$ in the range $\{1,\ldots,i_1-1\}$. Moreover, for $2\leq k \leq t-h-1$ and fixed $i_{k}=a_{d_{t-k+1}+1}$, there are $\binom{i_{k}-i_{k-1}-1}{d_{t-k+1}-d_{t-k+2}-1}$ ways to choose the values $a_{d_{t+k+2}+2},\ldots,a_{d_{t+k+1}}$ in the range $\{i_{k-1}+1,\ldots,i_{k}-1\}$.  Furthermore, there are $\binom{\widetilde{m}-i_{t-h-1}}{d_{h+1}-d_{h+2}-1}$ ways to choose $a_{d_{h+2}+2},\ldots, a_{d_{h+1}}$ in the range $\{i_{t-h-1}+1,\ldots, \widetilde{m}\}$. \\
For fixed $j_1=b_{c_{h}-c_{h-1}}$  there are $\binom{j_1-1}{c_{h}-c_{h-1}-1}$ ways to choose the values $b_{1},\ldots,b_{c_{h}-c_{h-1}-1}$ in the range $\{1,\ldots,j_1-1\}$. Moreover, for $2\leq k \leq h$ and fixed $j_{k}= b_{c_h-c_{h-k}}$ there are $\binom{j_{k}-j_{k-1}-1}{c_{h-k+1}-c_{h-k}-1}$ ways to choose the values $b_{c_{h}-c_{h-k+1}+1}$, $\ldots$, $b_{c_{h}-c_{h-k}-1}$ in the range $\{j_{k-1}+1,\ldots,j_{k}-1\}$. Finally, there are $\binom{\widetilde{m}+\widetilde{n}-j_{h-1}}{c_1}$ ways to choose $b_{c_{h}-c_1+1},\ldots, b_{c_{h}}$ in the range $\{j_{h-1}+1,\ldots, \widetilde{m}+\widetilde{n}\}$.
Hence in this case, we conclude that $\mathrm{type}(K[\Pc])$ is $A_{h} \cdot B_{h}$.
Observe that the formula for the $a_{i}$ makes sense only if $1 \leq h\leq t-2$. For $h=t-1$, we have to choose the numbers
\[
a_{1},\ldots,a_{d_{t}}
\]
among the values $\{1,\ldots, m-c_{t-1}\}$, hence $A_{t-1}=\binom{m-c_{t-1}}{d_t}$. Furthermore observe that the formula for the $b_{i}$ makes sense only if $2 \leq h\leq t-1$. For $h=1$, we have to choose the numbers
\[
b_{1},\ldots,b_{c_1}
\]
among the values $\{m-c_{1}+1,\ldots, (m-c_{1})+(n-d_{2}) \}$, hence $B_{1}=\binom{n-d_{2}}{c_1}$. 
\end{proof}

We observe that the conditions of Theorem \ref{gorenstein} also arise from the above formula. 
%In fact, if $\Pc$ is Gorenstein, then all of the maximal chains have same length, in particular $m=n=d_{h+1}+c_{h}$ for any $h=1,\ldots, t-1$. We have that $\Pc_{1}$ has bounding box of size $c_{t-1}\times n-d_{t}$, hence it is a square. Similarly also $\Pc_{2},\ldots,\Pc_{t}$. 
The following example demonstrates Theorem \ref{thm:type}.
\begin{Example}
Let $\Pc$ be the Ferrer diagram in Figure \ref{fig:Ex}. 
\begin{figure}[H]
   \centering
  \resizebox{0.4\textwidth}{!}{
  \begin{tikzpicture}

\draw (0,4)--(5,4);
\draw(0,3)-- (5,3);
\draw(0,2)-- (3,2);
\draw(0,1)-- (2,1);
\draw(0,0)-- (1,0);

\draw (0,0)--(0,4);
\draw (1,0)--(1,4);
\draw (2,1)--(2,4);
\draw (3,2)--(3,4);
\draw (4,3)-- (4,4);
\draw (5,3)-- (5,4);

%\fill[fill=gray, fill opacity=0.2] (1,2) -- (2,2)-- (2,3) -- (1,3);

\filldraw (1,0)  circle (1.5*\rad) node[anchor=north] {$V_1$}; 
\filldraw (2,1)  circle (1.5*\rad) node[anchor=north] {$V_2$}; 
\filldraw (3,2)  circle (1.5*\rad) node[anchor=north] {$V_3$}; 
\filldraw (4,3)  circle (1.5*\rad) node[anchor=north] {$V_4$}; 
\filldraw (5,3)  circle (1.5*\rad) node[anchor=north] {$V_5$}; 
\filldraw (0,1)  circle (1.5*\rad) node[anchor=east] {$H_1$}; 
\filldraw (0,2)  circle (1.5*\rad) node[anchor=east] {$H_2$}; 
\filldraw (0,3)  circle (1.5*\rad) node[anchor=east] {$H_3$}; 
\filldraw (0,4)  circle (1.5*\rad) node[anchor=east] {$H_4$}; 

   \end{tikzpicture}}
   \caption{}\label{fig:Ex}
\end{figure}
We have $t=4$ maximal rectangles whose sizes are $\{c_i\times d_i\}_{i=1,\ldots,4}$ with
\begin{table}[H]
\begin{tabular}{cccc}
$c_{1}=1$ &$c_{2}=2$ &$c_{3}=3$ &$c_{4}=m=5$ \\
$d_{1}=n=4$ &$d_{2}=3$ &$d_{3}=2$ &$d_{4}=1.$
\end{tabular}
\end{table}
There are $4$ maximal chains in the poset $Q$ corresponding to $\Pc$ containing $5$ vertices. For example, the chain $V_{1},\ldots,V_{5}$ and the chain $H_{1},H_{2},V_{3},V_{4},V_{5}$, that correspond to the cases $r=m$ and $r=(n-d_{3})+(m-c_{2})$, hence $h=2$.
We are going to compute $A$ and $A_{2}B_{2}$ as in Theorem \ref{thm:type}.
We have
\[
A=\sum\limits_{i_1=3}^3\sum\limits_{i_2=4}^4 \sum\limits_{i_3=5}^5  \binom{i_1-1}{1} \binom{i_2-i_{1}-1}{2-1-1} \binom{i_3-i_{2}-1}{3-2-1} \binom{5-i_3}{4-3-1}=2,
\]
while
\[
A_{2}=\sum\limits_{i_1=3}^3 \binom{i_1-1}{1}\binom{5-2-i_1}{2-1-1}=2
\]
and
\[
B_{2}=\sum\limits_{i_1=4}^4 \binom{i_1-1}{2-1-1} \binom{5-i_1}{1}=1,
\]
yielding
\[
A_{2}B_{2}=2.
\]
\end{Example}


\begin{thebibliography}{}

\bibitem{Bi}{G. Birkhoff},
\textit{Rings of sets},
\newblock Duke Math. J.,3, 443--454 (1937).

\bibitem{BV}{W. Bruns, U.Vetter}, \textit{Determinantal rings},
\newblock Lecture Note in Mathematics 1327, (1988). 


\bibitem{CGG}{L. Caniglia, J. A. Guccione, J. J. Guccione}, \textit{Ideals of generic minors}, 
\newblock Comm. Algebra,  18,  2633--2640,  (1990. ,


\bibitem{CFRR}{G. Castiglione, A. Frosini, A. Restivo, S. Rinaldi},
\textit{Tomographical aspects of L-convex Polyominoes},
\newblock Pure Math. Appl., 18, 239--256, (2007).

\bibitem{Co1}{A. Conca}\textit{ Ladder determinantal rings},  J.   Pure and Appl.  Alg.,   98,  119--134,   (1995).


\bibitem{Co2}{A. Conca}\textit{ Gorenstein ladder determinantal rings},  J. London Math.  Soc.,  54, 453--474, (1996).

\bibitem{CH}{A. Conca, J. Herzog}
\textit{Ladder determinantal rings have rational singularities}
\newblock Adv. Math., 132,  120--147, (1997). 

\bibitem{CN}{A. Corso, U. Nagel}, 
\textit{Monomial and toric ideals associated to Ferrers graphs}, 
\newblock Trans. Amer. Math. Soc. 361, 1371--1395, (2009).


\bibitem{CR}{G. Castiglione, A. Restivo},
\textit{Reconstruction of L-convex Polyominoes},
\newblock Electron. Notes Discrete Math., 12, 290--301, (2003).

\bibitem{GG}{C. D. Godsil, I. Gutman},
\textit{Some remarks on matching polynomials and its zeros}, 
\newblock Croatica Chemica Acta, 54, 53--59, (1981).

\bibitem{Gor}{E. Gorla}{\textit Mixed ladder determinantal varieties from two-sided ladders},
\newblock  J.  Pure and Appl. Alg.,   211.2 (2007): 433-444.


\bibitem{HH}{J. Herzog, T. Hibi},
\textit{Monomial Ideals},
\newblock Graduate Texts in Mathematics, Springer, (2010).


\bibitem{HM}{J. Herzog,  S. Saeedi  Madani},  {\textit The coordinate ring of a simple polyomino},
\newblock  Illinois J. Math.,   58,  981--995, (2014).

\bibitem{HMP}{J. Herzog, F. Mohammadi, J. Page},
\textit{Measuring the non-Gorenstein locus of Hibi rings and normal affine semigroup rings},
\newblock J. Algebra, 540, 78--99, (2019).

\bibitem{Hi}{T. Hibi},
\textit{Distributive lattices, affine semigroup rings and algebras with straightening laws},
\newblock Commutative Algebra and Combinatorics, Advanced Studies in Pure Math., M. Nagata and H. Matsumura, (Eds.), Vol. 11,
North--Holland, Amsterdam, 93--109, (1987).

\bibitem{HO}{T. Hibi, H. Ohsugi },
\textit{Koszul bipartite graphs},
\newblock Adv. Appl. Math., 22, 25--28, (1999).

\bibitem{MRR}{C. Mascia, G. Rinaldo, F.Romeo},
\textit{Primality of multiply connected polyominoes},
\newblock preprint arXiv:1907.08438, 1--16,(2019).

\bibitem{Mi}{M. Miyazaki}, 
\textit{On the generators of the canonical module of a Hibi ring: a criterion of level property and the degrees of generators},
\newblock J. Algebra, 480, 215--236,  (2017).

\bibitem{Na}{H. Narasimhan}, \textit{The irreducibility of ladder determinantal varieties}
\newblock J. Algebra,  102, 162--185, (1986). 



\bibitem{Q}{A. A. Qureshi},
\textit{Ideals generated by 2-minors, collections of cells and stack polyominoes},
\newblock J. Algebra, 357, 279--303, (2012).


\bibitem{QShShi}{A. A. Qureshi, T.  Shibuta, A.  Shikama} 
\textit{Simple polyominoes are prime},
\newblock J.  Commutative  Alg.,  9,  413--422, (2017). 

\bibitem{St}{R.P. Stanley},
\textit{Hilbert functions of graded algebras},
\newblock Adv. Math. 28, 57--83, (1978).


\bibitem{Stu}{B. Sturmfels}, \textit{Gr\"obner bases and Stanley decompositions of determinantal rings},
\newblock Math. Z.,  205, 137--144, (1990).  


\bibitem{Tu}{H. Tulleken},
\textit{Polyominoes 2.2. How they fit together},
\newblock  Online Edition (2019)

\end{thebibliography}
\end{document}